\documentclass[12pt]{amsart}
\usepackage{amssymb,latexsym,mathtools}
\usepackage[margin=1.25in]{geometry}
\usepackage{mathrsfs}
\usepackage{mathtools}
\usepackage{graphicx}
\usepackage{esint}
\usepackage{braket}
\newtheorem{theorem}{Theorem}
\newtheorem{lemma}{Lemma}
\newtheorem{definition}{Definition}

\newtheorem{prop}{Proposition}

\newtheorem{remark}{Remark}
\newtheorem{corollary}{Corollary}

\begin{document}

\title[Surface Tension and $\Gamma$-Convergence]{Surface Tension and $\Gamma$-Convergence of Van der Waals-Cahn-Hilliard Phase Transitions in Stationary Ergodic Media}
\author[P.\ Morfe]{Peter S.\ Morfe}

\maketitle

\begin{abstract} We study the large scale equilibrium behavior of Van der Waals-Cahn-Hilliard phase transitions in stationary ergodic media.  Specifically, we are interested in free energy functionals of the following form
\begin{equation*}
\mathcal{F}^{\omega}(u) = \int_{\mathbb{R}^{d}} \left(\frac{1}{2} \varphi^{\omega}(x,Du(x))^{2} + W(u(x)) \right) \, dx,
\end{equation*}
where $W$ is a double-well potential and $\varphi^{\omega}(x,\cdot)$ is a stationary ergodic Finsler metric.  
We show that, at large scales, the random energy $\mathcal{F}^{\omega}$ can be approximated by the anisotropic perimeter associated with a deterministic Finsler norm $\tilde{\varphi}$.  To find $\tilde{\varphi}$, we build on existing work of Alberti, Bellettini, and Presutti, showing, in particular, that there is a natural sub-additive quantity in this context.
\end{abstract}

\section{Introduction}

\subsection{Overview}  In this work, we analyze the large scale equilibrium behavior of Van der Waals-Cahn-Hilliard phase transitions in stationary ergodic media.  The free energy is described by functionals of the form
\begin{equation} \label{E: energy}
\mathcal{F}^{\omega}(u) = \int_{\mathbb{R}^{d}} \left(\frac{1}{2} \varphi^{\omega}(x,Du(x))^{2} + W(u(x)) \right) \, dx,
\end{equation}
where the argument $u$ is a scalar function taking values in $[-1,1]$, $W$ is a double well potential with wells of equal depth, and $\varphi^{\omega}(x,\cdot)$ is a stationary ergodic Finsler metric.  

Functionals like \eqref{E: energy} arise in diffuse-interface or phase-field models of phase transitions (cf.\ the original papers \cite{van der Waals}, \cite{cahn hilliard}, \cite{allen cahn}, the surveys \cite{langer}, \cite{widom}, and the book \cite{diffuse interface book}) and are closely related to Ising spin systems.  In this phenomenological mesoscopic scale theory, $u(x)$ describes the state of the material at the point $x$ and the two minima of $W$ correspond to the equilibrium states.  The gradient term in \eqref{E: energy} imposes a penalty on configurations $u$ that transition between the two equilibria.  Effectively, this puts a constraint on the shape of energy minimizing configurations.  

We are interested in the role played by the randomness in determining the shape of the transition regions of energy minimizers in large domains, particularly whether or not averaging effects allow these to be described by an effective macroscopic model.  


Classically, in spatially homogeneous media, that is, when $\varphi^{\omega}(x,\cdot)$ does not depend on $x$ or $\omega$, Modica and Mortola \cite{first paper} proved that the leading order behavior of $\mathcal{F}^{\omega}$ at large scales is described by a surface energy functional measuring the transition region.  Roughly speaking, if the transition region of $u$ is concentrated in a neighborhood of $RF$, where $F \subseteq \mathbb{R}^{d}$ is a hypersurface and $R > 0$ is large, then the energy of $u$ is at least $R^{d - 1} \mathscr{E}(F)$, where $\mathscr{E}$ is the functional defined by 
\begin{equation} \label{E: surface energy}  
\mathscr{E}(F) = \int_{\partial F} \tilde{\varphi}(\nu_{F}(\xi)) \, \mathcal{H}^{d -1}(d \xi)
\end{equation}
and $\tilde{\varphi}$ is a one-homogeneous convex function we will refer to as the surface tension.  
In effect, at large scales, the shape of the transition region of a minimizing configuration is related to the minimal surfaces of \eqref{E: surface energy}.  As long as we are working in the spatially homogeneous context, $\tilde{\varphi}$ equals $\varphi$ scaled by a constant depending only on $W$.  

In stationary ergodic media, we prove below that a result of this type still holds except that $\tilde{\varphi}$ now depends on the the average behavior of $\mathcal{F}^{\omega}$.

The proof in spatially homogeneous media is essentially one-dimensional in nature, and the surface tension can be characterized via the so-called optimal profile problem:
\begin{equation*}
\tilde{\varphi}(e) = \min \left\{ \int_{-\infty}^{\infty} \left( \frac{1}{2} \varphi(e)^{2} q'(s)^{2} + W(q(s)) \right) \, ds \, \mid \, \lim_{s \to \pm \infty} q(s) = \pm 1 \right\}.
\end{equation*} 
The point is once we understand the best one-dimensional configurations, it is possible to estimate the energy of an arbitrary configuration $u$ by studying its one-dimensional slices (see \cite{alberti guide} and \cite{free discontinuity problems}).  Here it is essential that the functional \eqref{E: energy} does not change when we translate space.   

In fact, in the spatially homogeneous setting, the solutions of the optimal profile problem carry a lot of information.  When $W$ is sufficiently regular and $\varphi$ is Riemannian, the optimal profile problem has a unique minimizer, often called the standing wave.  The standing wave is useful not only in the study of the asymptotics of the free energy \eqref{E: energy} but also in understanding the long-time behavior and scaling limit of its gradient flow (cf.\ \cite{fife mcleod}, \cite{crazy de masi presutti paper}, \cite{barles souganidis}, \cite{scaling limits}). 

Although in periodic media plane-like minimizing configurations appear to be the natural replacement of the standing wave (cf.\ \cite{valdinoci de la llave} and the references therein), in random media it is not clear that plane-like minimizers exist, and, in fact, there are examples to show that they do not exist in general when $d = 1$.  Nonetheless, in a slightly different context, it has been known for some time that the large scale behavior of functionals like \eqref{E: energy} can be understood without resorting to slicing arguments or computations involving plane-like minimizers.  In their study of the (spatially homogeneous) Lebowitz-Penrose functional, Alberti and Bellettini showed that it is possible to find the surface tension by studying the optimal energy of planar interfaces on infinite cylinders \cite{non-local_anisotropic}.  

Later, a similar approach was used by Ansini, Braides, and Chiad\`{o} Piat \cite{periodic_paper}, who proved that, in periodic media, $\mathcal{F}^{\omega}$ can still be described by an effective surface tension.  The formula for the surface tension obtained in \cite{periodic_paper} involves averages of the energy of planar configurations in larger and larger cubes, which converge because the quantities involved are sub-additive up to an asymptotically negligible error.  In the stationary ergodic context, similar formulas have been derived recently for related interface models by Alicandro, Cicalese, and Ruf  \cite{magnetic domains} and Cagnetti, Dal Maso, Scardia, and Zeppieri \cite{stochastic homogenization free discontinuity}.  Both of these works find formulas for the surface tension using slight modifications of the quantities considered in \cite{periodic_paper} that turn out to be genuinely sub-additive and, thus, can be analyzed using the sub-additive ergodic theorem.  The fundamental difficulty \cite{magnetic domains} and \cite{stochastic homogenization free discontinuity} share in common with the present  work is it is necessary to find a way to apply the sub-additive ergodic theorem to a quantity that scales like area rather than volume.  

Here that difficulty is resolved by revisiting the ideas of Alberti and Bellettini.  Although they did not mention it explicitly, sub-additivity is lurking in the background in their formula for the surface tension.  This is clarified in Presutti's book \cite{scaling limits}, where it is implicit in his arguments.  

We extend Presutti's interpretation of Alberti and Bellettini's formula and use it to find the surface tension of \eqref{E: energy} in stationary ergodic media.  Our approach is both similar and rather different than that in \cite{magnetic domains} and \cite{stochastic homogenization free discontinuity}.  In particular, what we prove clarifies the role played by sub-additivity in problems of this type and shows that averaging over cylinders is in some ways more natural than averaging over cubes.

Together with \cite{magnetic domains} and \cite{stochastic homogenization free discontinuity}, the present paper can be seen as an extension of Dal Maso and Modica's variational framework (see \cite{nonlinear stochastic homogenization}) to stochastic homogenization problems in an SBV setting.  As in the Sobolev space case, there are compactness theorems and integral representation formulas for sufficiently nice functionals on SBV.  Moreover, there are results to the effect that such functionals are uniquely determined by their values on interfaces with planar boundary traces, the SBV analogue of linear boundary data.  In the setting considered here, \cite{periodic_paper} provides both of these ingredients.  Therefore, to carry out Dal Maso and Modica's program, it only remains to show that planar configurations homogenize, or, in other words, to determine the macroscopic surface tension $\tilde{\varphi}$.

%

Before proceeding to our assumptions and main results, we briefly review a number of other recent treatments of interface energies in random media.  Dirr and Orlandi \cite{dirr} consider an energy functional similar to ours, except that the potential $W$ is random instead of the metric.  Assuming that the random potential is a sufficiently small random perturbation of a deterministic one and that the randomness has a finite range of dependence, they studied not only the surface tension but also the global equilibria of the energy, an issue we avoid here.  In a discrete context, Gold \cite{gold} and Wouts \cite{wouts} study the surface tension of percolation and Ising models, respectively, again in random media with a finite range of dependence.  The basic strategy for obtaining the surface tension in those two papers is very similar to ours, except that they prove the convergence of the mean energy first and then use concentration inequalities to control the fluctuations.  Our work, together with \cite{magnetic domains} and \cite{stochastic homogenization free discontinuity}, suggests that, at least where the determination of the surface tension is concerned, there is generally no need to impose a finite correlation length or quantitative mixing assumptions on the underlying environment.  

\subsection{Assumptions}  Throughout we fix constants $\lambda, \Lambda \in (0,\infty)$.  We let $\mathscr{F}_{d}(\lambda,\Lambda)$ be the set of continuous functions $\varphi : S^{d - 1} \to [0,\infty)$ such that
\begin{itemize}
\item[(i)] $\sqrt{\lambda} \leq \varphi(e) \leq \sqrt{\Lambda}$ if $e \in S^{d -1}$
\item[(ii)] The one-homogeneous extension of $\varphi$ to $\mathbb{R}^{d}$ is convex
\end{itemize}
 We equip $\mathscr{F}_{d}(\lambda,\Lambda)$ with the Borel $\sigma$-algebra induced by the supremum norm topology.  
 
 We assume that $(\Omega,\mathscr{B},\mathbb{P})$ is a probability space equipped with a measurable map $\Phi: \Omega \to \mathscr{F}_{d}(\lambda,\Lambda)$ and an action $\tau : \mathbb{R}^{d} \times \Omega \to \Omega$. 
 
 Our assumptions on $\tau$ are:
 \begin{itemize}
 \item[(i)] $\tau$ is measurable with respect to the product $\sigma$-algebra on $\mathbb{R}^{d} \times \Omega$ obtained from the Lebesgue measurable sets on $\mathbb{R}^{d}$ and $\mathscr{B}$ on $\Omega$
 \item[(ii)] $\tau$ forms a group under composition, that is,
  \begin{equation*}
 \tau_{x + y} = \tau_{x} \circ \tau_{y} \quad \text{if} \, \, x, y \in \mathbb{R}^{d}, \quad \tau_{0} = \text{Id}.
 \end{equation*}
 \item[(iii)] $\tau$ preserves $\mathbb{P}$ in the following sense:
 \begin{equation*}
 \mathbb{P}(\tau_{x}^{-1}(E)) = \mathbb{P}(E) \quad \text{if} \, \, E \in \mathscr{B}, \, \, x \in \mathbb{R}^{d}.
 \end{equation*}
 \item[(iv)] $\tau$ is ergodic: specifically, if $E \in \mathscr{B}$ and $\tau_{x}^{-1}(E) = E$ independently of the choice of $x \in \mathbb{R}^{d}$, then $\mathbb{P}(E) \in \{0,1\}$.  
 \end{itemize}
 We remark that the ergodicity assumption (iv) is not strictly necessary.  As in \cite{magnetic domains} and \cite{stochastic homogenization free discontinuity}, if we remove (iv), then a version of our result still holds with the caveat that the surface tension remains a random process in the limit. 
 
The Finsler metric $\varphi^{\omega} : \mathbb{R}^{d} \times \mathbb{R}^{d} \to [0,\infty)$ appearing in \eqref{E: energy} is related to $\Phi$ and $\tau$ by 
 \begin{equation*}
 \varphi^{\omega}(x,p) = \left\{ \begin{array}{r l}
 						\|p\| \Phi^{\tau_{x}\omega} \left(\frac{p}{\|p\|}\right), & p \in \mathbb{R}^{d} \setminus \{0\} \\
						0, & \text{otherwise}
					\end{array} \right.
 \end{equation*}
 Notice that $\varphi^{\omega}(x,\cdot)$ is convex for each fixed $x \in \mathbb{R}^{d}$, and $\varphi^{\omega}(\cdot,p)$ is Lebesgue measurable for each $p \in \mathbb{R}^{d}$.  
Moreover, $\varphi^{\omega}$ is stationary in the following manner:
 \begin{equation*}
 \varphi^{\tau_{y} \omega}(x,\cdot) = \varphi^{\omega}(x + y,\cdot) \quad \text{if} \, \, x,y \in \mathbb{R}^{d}.
 \end{equation*}  
 
Regarding the potential $W$ in \eqref{E: energy}, we only assume the following:
 \begin{itemize}
 \item[(i)] $W : [-1,1] \to [0,\infty)$ is continuous
 \item[(ii)] $W^{-1}(\{0\}) = \{-1,1\}$
 \end{itemize}

\subsection{Main results}  Our results are stated in the language of $\Gamma$-convergence.  See the textbooks of Dal Maso \cite{dal maso} or Braides \cite{gamma for beginners} for an introduction to the subject.

Before proceeding further, we define rescaled energy functionals $(\mathcal{F}^{\omega}_{\epsilon})_{\epsilon > 0}$ and introduce the notion of localization.  First, for each $\epsilon > 0$, we define $\mathcal{F}^{\omega}_{\epsilon}$ by 
\begin{equation*}
\mathcal{F}^{\omega}_{\epsilon}(u) = \int_{\mathbb{R}^{d}} \left( \frac{\epsilon}{2} \varphi^{\omega}(\epsilon^{-1} x,Du(x))^{2} + \epsilon^{-1} W(u(x)) \right) \, dx. 
\end{equation*}
The reader can check that $\mathcal{F}^{\omega}_{\epsilon}$ is obtained from $\mathcal{F}^{\omega}$ by rescaling space linearly and renormalizing by $\epsilon^{d - 1}$.  In other words, $\mathcal{F}^{\omega}_{\epsilon}$ is the functional obtained when we set the characteristic length scale of the system to $\epsilon$, and the exponent $d -1$ reflects the fact that the energy scales like surface area.   

In addition to rescaling, it is convenient to localize the functionals $(\mathcal{F}^{\omega}_{\epsilon})_{\epsilon > 0}$ by associating to each open subset $A \subseteq \mathbb{R}^{d}$ its own energy:
\begin{equation*}
\mathcal{F}^{\omega}_{\epsilon}(u; A) = \int_{A} \left(\frac{\epsilon}{2} \varphi^{\omega}(\epsilon^{-1} x,Du(x))^{2} + \epsilon^{-1} W(u(x)) \right) \, dx.
\end{equation*}  
In the sequel, the case $\epsilon = 1$ is frequently of interest and we will abbreviate by writing $\mathcal{F}^{\omega}(\cdot ; A) = \mathcal{F}^{\omega}_{1}(\cdot; A)$.  

Though at the macroscopic scale the objects of interest are surfaces, it is customary to treat them as functions.  Note that if $F \subseteq \mathbb{R}^{d}$ is a Caccioppoli set, then the function $u$ satisfying $u(x) = 1$ if $x \in F$ and $u(x) = -1$, otherwise, uniquely determines $F$.  Conversely, any function $u \in \text{BV}_{\text{loc}}(\mathbb{R}^{d}; \{-1,1\})$ is associated to such a set, namely $F = \{u = 1\}$.  Accordingly, in what follows, we will treat $\mathscr{E}$ as a functional defined on BV functions rather than Caccioppoli sets.  Specifically, if $A$ is a bounded open set, we write
\begin{equation*}
\mathscr{E}(u; A) = \left\{ \begin{array}{r l}
		\int_{\partial^{*} \{u = 1\} \cap A} \tilde{\varphi}(\nu_{\{u = 1\}}(\xi)) \, \mathcal{H}^{d - 1}(d \xi), & u \in BV(A; \{-1,1\}) \\
		\infty, & \text{otherwise}.
		\end{array} \right.
\end{equation*}

We are now prepared to state our main result.  In the statement of the next theorem, $\Sigma$ is the $\sigma$-algebra consisting of those events $A \in \mathscr{B}$ such that $\tau_{x}^{-1}(A) = A$ no matter the choice of $x \in \mathbb{R}^{d}$.  

\begin{theorem} \label{T: main result}  There is a one-homogeneous, convex function $\tilde{\varphi} : \mathbb{R}^{d} \to (0,\infty)$ depending only on $\mathbb{P}$ such that, with probability one, $\mathcal{F}^{\omega} \overset{\Gamma}\to \mathscr{E}$.  More specifically, there is an event $\hat{\Omega} \in \Sigma$ such that $\mathbb{P}(\hat{\Omega}) = 1$ and no matter the choice of Lipschitz, open, bounded $A \subseteq \mathbb{R}^{d}$ or $\omega \in \hat{\Omega}$, the following occurs:
\begin{itemize}
\item[(i)] If $(u_{\epsilon})_{\epsilon > 0} \subseteq H^{1}(A; [-1,1])$ satisfies
\begin{equation*}
\sup \left\{ \mathcal{F}^{\omega}_{\epsilon}(u_{\epsilon}; A) \, \mid \, \epsilon > 0\right\} < \infty,
\end{equation*}
then $(u_{\epsilon})_{\epsilon > 0}$ is relatively compact in $L^{1}(A)$ and all of its limit points are in $BV(A;\{-1,1\})$.  
\item[(ii)] If $u \in L^{1}(A;[-1,1])$ and $(u_{\epsilon})_{\epsilon > 0} \subseteq H^{1}(A;[-1,1])$ satisfies $u_{\epsilon} \to u$ in $L^{1}(A)$, then 
\begin{equation*}
\mathscr{E}(u; A) \leq \liminf_{\epsilon \to 0^{+}} \mathcal{F}^{\omega}_{\epsilon}(u_{\epsilon}; A).
\end{equation*}
\item[(iii)] If $u \in L^{1}(A;[-1,1])$, then there is a family $(u_{\epsilon})_{\epsilon > 0} \subseteq H^{1}(A; [-1,1])$ such that $u_{\epsilon} \to u$ in $L^{1}(A)$ and
\begin{equation*}
\limsup_{\epsilon \to 0^{+}} \mathcal{F}^{\omega}_{\epsilon}(u_{\epsilon};A) \leq \mathscr{E}(u; A).
\end{equation*}
\end{itemize}  \end{theorem}  

We emphasize that the event $\hat{\Omega}$ in the theorem does not depend on the domain $A$.  

Our choice to consider functions taking values in $[-1,1]$ is one of convenience.  It is possible to generalize Theorem \ref{T: main result} to functions taking values in $\mathbb{R}$ instead of $[-1,1]$ provided one imposes certain growth assumptions on $W$ (see \cite{periodic_paper}).  Since the transformation $u \mapsto (u \vee -1) \wedge 1$ can only decrease the energy, the general case is easily recovered from our work.  

Theorem \ref{T: main result} is proved by finding the surface tension $\tilde{\varphi}$ and then applying the compactness and integral representation results of \cite{periodic_paper}.  We will now sketch how $\tilde{\varphi}$ is found.  

Following \cite{periodic_paper}, we fix an absolutely continuous function $q : \mathbb{R} \to [-1,1]$ satisfying
\begin{equation} \label{E: BC}
\left\{ \begin{array}{r l}
			\int_{-\infty}^{\infty} \left(\frac{\Lambda}{2} q'(s)^{2} + W(q(s)) \right) \, ds < \infty \\
			\lim_{s \to \pm \infty} q(s) = \pm 1 
	\end{array} \right.
\end{equation}
For a given $e \in S^{d - 1}$ and $x \in \mathbb{R}^{d}$, we define $q_{e}, T_{x} q_{e} : \mathbb{R}^{d} \to [0,\infty)$ by $q_{e}(y) = q(\langle y, e \rangle)$ and $T_{x}q_{e}(y) = q_{e}(y - x)$.  

We use the function $T_{x}q_{e}$ to build nearly optimal phase configurations with asymptotically flat transition regions that are centered around the affine hyperplane $\{\langle y - x, e \rangle = 0\}$.  These functions are built by perturbing $T_{x}q_{e}$ in progressively larger domains.  We then obtain $\tilde{\varphi}(e)$ by suitably rescaling the energy of these configurations and letting the domain of the perturbation fill the whole space.  

To make this precise, we proceed by analogy with statistical mechanics.  To start with, we define a random process $\tilde{\Phi}^{\omega}$ that we call the finite-volume surface tension.  If $e \in S^{d -1}$, $x_{0} \in \mathbb{R}^{d}$, and $A \subseteq \mathbb{R}^{d}$ is a bounded open set, $\tilde{\Phi}^{\omega}(e,x_{0},A)$ is defined by the following formula:
\begin{equation*}
\tilde{\Phi}^{\omega}(e,x_{0},A) = \min \left\{ \mathcal{F}^{\omega}(u; A) \, \mid \, u \in H^{1}(A; [-1,1]), \, \, u - T_{x}q_{e} \in H^{1}_{0}(A) \right\}.
\end{equation*}
Physically, $\tilde{\Phi}^{\omega}(e,x_{0},A)$ is the optimal energy achievable when we perturb the planar configuration $T_{x}q_{e}$ in $A$.  We will show that the surface tension $\tilde{\varphi}$ can be obtained by studying $\tilde{\Phi}^{\omega}$.

It is convenient to begin by restricting to the case when $x_{0} = 0$.  First, a bit of notation.  Henceforth, fix a direction $e \in S^{d - 1}$ and let $O_{e} : \mathbb{R}^{d - 1} \to \mathbb{R}^{d}$ be a linear isometry onto the hyperplane $\langle e \rangle^{\perp}$ orthogonal to $e$.  Here and in the sequel, if $A \subseteq \mathbb{R}^{d - 1}$ and $B \subseteq \mathbb{R}$, then we define $A \oplus_{e} B \subseteq \mathbb{R}^{d}$ by 
\begin{equation*}
A \oplus_{e} B = \left\{ O_{e}(y) + te \, \mid \, y \in A, \, \, t \in B \right\}.
\end{equation*}
Additionally, we fix an orthonormal basis of $\mathbb{R}^{d - 1}$ and let $Q(0,R)$ denote the cube in $\mathbb{R}^{d - 1}$ centered at the origin, oriented according to this basis, and with side length $R > 0$.  

The centered finite-volume surface tension in the $e$ direction is the random process defined as follows: if $A \subseteq \mathbb{R}^{d - 1}$ is a bounded open set and $h > 0$, then we write
\begin{equation*}
\tilde{\varphi}^{\omega}(e,A,h) = \tilde{\Phi}^{\omega}(e,0,A \oplus_{e} (-h,h)).
\end{equation*}
Below we will demonstrate the utility of studying $\tilde{\Phi}^{\omega}$ as a function of the $e$ directions and the orthogonal directions separately.  For now, we remark that $\tilde{\varphi}^{\omega}(e,A,h)$ can be understood as a sub-additive process in $A$ and an almost monotone function in $h$.  

As a consequence of this monotonicity, it is possible to show that $\tilde{\varphi}^{\omega}(e,A,h)$ has a limit as $h \to \infty$.  In fact, if we let $\tilde{\varphi}^{\omega}_{\infty}(e,A) = \lim_{h \to \infty} \tilde{\varphi}^{\omega}(e,A,h)$, then the following formula holds:
\begin{equation} \label{E: infinite_horizon_intro}
 \tilde{\varphi}^{\omega}_{\infty}(e,A) = \min \left\{ \mathcal{F}^{\omega}(u; A \oplus_{e} \mathbb{R}) \, \mid \, -1 \leq u \leq 1, \, \, u - q = 0 \, \, \text{on} \, \, \partial A \oplus_{e} \mathbb{R} \right\}.
\end{equation}
Evidently, $\tilde{\varphi}^{\omega}_{\infty}(e,A)$ inherits the sub-additivity property in the limit $h \to \infty$, although it also follows by inspection of the right-hand side of \eqref{E: infinite_horizon_intro}.  Notice that here we are studying the optimal energy in a cylinder as opposed to a cube.

The surface tension can be obtained from $\tilde{\Phi}^{\omega}$ in a procedure we refer to as the thermodynamic limit.  In this context, this means both letting the set $A$ fill up $\mathbb{R}^{d- 1}$ and sending $h \to \infty$.  The result is summarized in the next theorem.  In the statement, $\Sigma_{e}$ is the $\sigma$-algebra of events $A \in \mathscr{B}$ such that $\tau_{x}^{-1}(A) = A$ if $x \in \langle e \rangle^{\perp}$.  

\begin{theorem} \label{T: thermodynamic_limit} For each $e \in S^{d - 1}$, there is an event $\tilde{\Omega}_{e} \in \Sigma_{e}$ satisfying $\mathbb{P}(\tilde{\Omega}_{e}) = 1$ such that if $\omega \in \tilde{\Omega}_{e}$ and $\kappa > 0$, then
\begin{align*}
\tilde{\varphi}(e) &= \lim_{R \to \infty} R^{1 - d} \tilde{\varphi}^{\omega}_{\infty}(e,Q(0,R)) \\
			&= \lim_{h \to \infty} \limsup_{R \to \infty} R^{1 - d} \tilde{\varphi}^{\omega}(e,Q(0,R),h) \\
			&= \lim_{h \to \infty} \liminf_{R \to \infty} R^{1 - d} \tilde{\varphi}^{\omega}(e,Q(0,R),h) \\
			&= \lim_{R \to \infty} R^{1 - d} \tilde{\varphi}^{\omega}(e,Q(0,R), \kappa R).
\end{align*}
\end{theorem}  

Theorem \ref{T: thermodynamic_limit} shows that the $e$ direction and the directions in $\langle e \rangle^{\perp}$ play distinct but complementary roles in the determination of the surface tension.  The limit $R \to \infty$ is handled using the sub-additive ergodic theorem.  This is intuitive: as $R$ grows, if we imagine that the developing transition region is roughly flat with normal vector $e$, then we are accommodating more and more surface area, and we expect that the fluctuations should average out.  On the other hand, the limit $h \to \infty$ does not require any rescaling, and, indeed, the arguments involved are completely deterministic in nature.  Essentially, the almost monotone dependence on $h$ reflects the fact that we are allowing more and more configurations as $h$ grows without losing any in the process.  Surprisingly, the boundary condition gives us enough control that all the limits in Theorem \ref{T: thermodynamic_limit} coincide.

We remark that the existence of the first limit in Theorem \ref{T: thermodynamic_limit} is an immediate consequence of the sub-additive ergodic theorem.  In this sense, it is very convenient to use infinite cylinders (see \eqref{E: infinite_horizon_intro}) instead of cubes.  On the other hand, the limit $R \to \infty$ along the line $h = R$ is the essential ingredient in the proof of $\Gamma$-convergence.  The existence of this limit is non-trivial, an issue that our problem shares in common with those considered in \cite{magnetic domains} and \cite{stochastic homogenization free discontinuity}.  Here is where we use the other two limits: once we prove the first three equalities in Theorem \ref{T: thermodynamic_limit}, the fourth follows easily.  

Theorem \ref{T: thermodynamic_limit} and its proof are the major difference between our work and the approach in \cite{magnetic domains} and \cite{stochastic homogenization free discontinuity}.  We comment on the relationship between their approach and ours in Remark \ref{R: response} below.  To sum it up, while our proof is longer, we feel it clarifies the role of the geometry of the problem, and it applies to arbitrary boundary conditions $q$ satisfying \eqref{E: BC}, not just flat ones.    

The statement and proof of Theorem \ref{T: thermodynamic_limit} is inspired by the analysis of the spatially homogeneous Lebowitz-Penrose functional in Chapter 7 of \cite{scaling limits}.  The novelty here is now the medium is heterogeneous and the approach needs to be supplemented by considerations from ergodic theory.

The proof we give below has in mind the case $d \geq 2$.  If $d = 1$, the same arguments work, but the dependence on $R$ above is superfluous.  Actually, in one dimension, the proof can be simplified considerably.  First of all, in that case, the surface tension can be computed in the following manner:
\begin{equation*}
\tilde{\varphi}(e) = \inf \left\{ \mathcal{F}^{\omega}(u; \mathbb{R}) \, \mid \, -1 \leq u \leq 1, \, \, \lim_{e x \to \pm \infty} u(x) = \pm 1 \right\} \quad \mathbb{P}\text{-almost surely}.
\end{equation*}
That the right-hand side should be constant follows from translational-invariance, which is not hard to check.  With this formula in hand, slight modifications of Alberti's proof of $\Gamma$-convergence in \cite{alberti guide} show that $\Gamma$-convergence holds.  Rather than expand on this here, we leave it as an exercise for the interested reader.  

Since  $R^{1 - d} \tilde{\varphi}^{\omega}(e,Q(0,R),R)$ converges as $R \to \infty$, Theorem \ref{T: thermodynamic_limit} effectively says that Theorem \ref{T: main result} holds provided $A$ is a cube centered in the hyperplane $\{\langle x, e \rangle = 0\}$ and $u = 2 \chi_{\{\langle x, e \rangle = 0\}} - 1$.  This is not immediate, but it is not a long proof either (see Proposition \ref{P: planar_convergence} below).  The extension to arbitrary cubes is established using the ergodic theorem and the fundamental estimate of $\Gamma$-convergence.  This is the content of our next theorem.

Before stating the theorem, we will need a bit of notation: if $\rho > 0$ and $x_{0} \in \mathbb{R}^{d}$ is decomposed as $x_{0} = O_{e}(\tilde{x}) + te$, then $Q^{e}(x_{0},\rho)$ is defined by 
\begin{equation*}
Q^{e}(x_{0},\rho) = Q(\tilde{x},\rho) \oplus_{e} \left(t - \frac{\rho}{2}, t + \frac{\rho}{2}\right).
\end{equation*}

\begin{theorem} \label{T: other_cubes}  There is a translationally-invariant event $\hat{\Omega} \in \Sigma$ satisfying $\mathbb{P}(\hat{\Omega}) = 1$ such that if $x_{0} \in \mathbb{R}^{d}$ and $\rho > 0$, then
\begin{equation*}
\tilde{\varphi}(e) \rho^{d - 1}= \lim_{R \to \infty} R^{1 - d} \tilde{\Phi}^{\omega}(e,Rx_{0},RQ^{e}(x_{0},\rho)) \quad \text{if} \, \, \omega \in \hat{\Omega}.
\end{equation*}
\end{theorem}  

Once Theorem \ref{T: other_cubes} is proved, $\Gamma$-convergence follows from the machinery developed in \cite{periodic_paper}.  Effectively, once we establish $\Gamma$-convergence of the energy of planar interfaces, the general case can be handled using purely deterministic arguments.

\subsection{Organization of the paper}  The analysis of the finite-volume surface tension is carried out in Section \ref{S: thermodynamic limit}.  We define the surface tension $\tilde{\varphi}$ and study the thermodynamic limit in Section \ref{S: surface tension}.  In the process, we also prove Theorem \ref{T: other_cubes} and continuity of $\tilde{\varphi}$.  The paper concludes in Section \ref{S: geometry}, where we prove Theorem \ref{T: main result}.  

Appendix \ref{A: fundamental_estimate} is dedicated to the statement and proof of the fundamental estimate of $\Gamma$-convergence as it pertains to \eqref{E: energy}.  A number of relevant results from ergodic theory are recalled in Appendix \ref{A: ergodic theory}. 

The notation is catalogued in Section \ref{S: assumptions/notations}.   

 \section{Notations} \label{S: assumptions/notations}
 
 \subsection{Euclidean Geometry} \label{S: euclidean} We let $\langle \cdot, \cdot \rangle : \mathbb{R}^{d} \times \mathbb{R}^{d} \to \mathbb{R}$ denote the standard Euclidean inner product in $\mathbb{R}^{d}$.  We use the notation $\|\cdot\|$ for the norm induced by $\langle \cdot, \cdot \rangle$.  
 
Given an $A \subseteq \mathbb{R}^{d}$, $A^{\perp}$ denotes the subspace of vectors in $\mathbb{R}^{d}$ orthogonal to every vector in $A$.  

If $v \in \mathbb{R}^{d}$, we let $\langle v \rangle$ denote its span in $\mathbb{R}^{d}$.  If $A \subseteq \mathbb{R}^{d}$ and $x \in \mathbb{R}^{d}$, then $x + A$ is the set obtained by translating $A$ by $x$, that is, 
\begin{equation*}
x + A = \{y + x \, \mid \, y \in A\}.
\end{equation*}  

The unit sphere in $\mathbb{R}^{d-1}$ is denoted by $S^{d - 1} = \{e \in \mathbb{R}^{d} \, \mid \, \|e\| = 1\}$.

We assume at the outset that we have fixed a family $\{O_{e}\}_{e \in S^{d - 1}}$ of linear maps $O_{e} : \mathbb{R}^{d - 1} \to \mathbb{R}^{d}$ preserving $\langle \cdot, \cdot \rangle$ so that $O_{e}(\mathbb{R}^{d - 1}) = \langle e \rangle^{\perp}$.  This is only a matter of notational convenience as our results are independent of this choice.


When $A \subseteq \mathbb{R}^{d - 1}$ and $B \subseteq \mathbb{R}$, we denote their direct sum along $e$ as $A \oplus_{e} B$, which is defined formally by
\begin{equation*}
A \oplus_{e} B = \left\{ O_{e}(a) + be \, \mid \, a \in A, \, \, b \in B \right\}. 
\end{equation*}
Similarly, if $x \in \mathbb{R}^{d - 1}$ and $s \in \mathbb{R}$, we define $x \oplus_{e} s = O_{e}(x) + se$.  

For each $k \in \{1,2,\dots,d\}$, we let $\mathcal{U}_{0}^{k}$ denote the collection of bounded open subsets of $\mathbb{R}^{k}$.

\subsection{Cubes} \label{S: cubes}

%

Throughout we fix an orthonormal basis $\{e_{1},\dots,e_{d - 1}\}$ of $\mathbb{R}^{d - 1}$.  We let $|\cdot|_{\infty} : \mathbb{R}^{d - 1} \to [0,\infty)$ denote the norm given by 
\begin{equation*}
|y|_{\infty} = \max\{|\langle y, e_{1} \rangle|, \dots, |\langle y, e_{d - 1} \rangle|\}.
\end{equation*}  
For each $r > 0$ and $y \in \mathbb{R}^{d - 1}$, $Q(y,r)$ is the open ball centered at $y$ with $|\cdot|_{\infty}$-radius $\frac{r}{2}$, that is,
\begin{equation*}
Q(y,r) = \left\{x \in \mathbb{R}^{d - 1} \, \mid \, |x - y|_{\infty} < \frac{r}{2} \right\}
\end{equation*}  
Note the factor of two, which simplifies many of the computations below.  

If $e \in S^{d - 1}$, $x \in \mathbb{R}^{d}$, and $r > 0$, then we denote by $Q^{e}(0,r)$ the open subset of $\mathbb{R}^{d}$ given by 
\begin{equation*}
Q^{e}(x,r) = x + Q(0,r) \oplus_{e} \left(-\frac{r}{2}, \frac{r}{2}\right).
\end{equation*}

%

\subsection{Geometric Measure Theory}  If $k \in \{1,2,\dots,d\}$, we let $\mathcal{L}^{k}$ be the Lebesgue measure on $\mathbb{R}^{k}$.  In $\mathbb{R}^{d}$, $\mathcal{H}^{k}$ is the $k$-dimensional Hausdorff measure normalized to coincide with $k$-dimensional area on smooth $k$-surfaces.  

If $E \subseteq \mathbb{R}^{d}$ is a Caccioppoli set, we denote by $\partial^{*}E$ its reduced boundary and $\nu_{E}$, its normal vector. 

\subsection{Functions}  If $f : \mathbb{R}^{d} \to \mathbb{R}$ is a function and $x \in \mathbb{R}^{d}$, the function $T_{x}f : \mathbb{R}^{d} \to \mathbb{R}$ is defined by
\begin{equation*}
(T_{x}f)(y) = f(y - x).
\end{equation*}

Given an $A \in \mathcal{U}_{0}^{d}$, we let $H^{1}(A)$ denote the closure of $C^{\infty}(A)$ with respect to the inner product $(\cdot,\cdot)_{H^{1}(A)}$ given by
\begin{equation*}
(f,g)_{H^{1}(A)} = \int_{A} \left(f(x) g(x) + \langle Df(x), Dg(x) \rangle \right) \, dx.
\end{equation*}
$H^{1}_{0}(A)$ denotes the closure of $C^{\infty}_{c}(A)$ in $H^{1}(A)$.  The subset $H^{1}(A;[-1,1])$ consists of those functions $u \in H^{1}(A)$ such that $-1 \leq u \leq 1$ $\mathcal{L}^{d}$-almost everywhere.

We define $\chi : \mathbb{R} \to \{-1,1\}$ by 
\begin{equation}
\chi(s) = \left\{ \begin{array}{r l}
				1, & s \geq 0 \\
				-1, & s < 0
			\end{array} \right.
\end{equation}

\subsection{Invariant $\sigma$-algebras}  We will frequently be interested in random variables that are invariant under certain subgroups of $(\tau_{x})_{x \in \mathbb{R}^{d}}$.  Given an $e \in S^{d -1}$, $\Sigma_{e}$ is the sub-$\sigma$-algebra of $\mathscr{B}$ generated by sets $A$ such that 
\begin{equation} \label{E: invariant_sets}
\tau_{x}^{-1}(A) = A \quad \text{if} \, \, x \in \langle e \rangle^{\perp}.
\end{equation}
We let $\Sigma = \bigcap_{e \in S^{d -1}} \Sigma_{e}$, i.e.\ $\Sigma$ is the family of sets for which \eqref{E: invariant_sets} holds no matter the choice of $x \in \mathbb{R}^{d}$.   

\subsection{Asymptotic Notation}  Given two functions $r, a : (0,\infty) \to \mathbb{R}$, we write $r(s) = o(a(s))$ if $\lim_{s \to 0^{+}} \frac{r(s)}{a(s)} = 0$.  In particular, the notation $r(s) = o(1)$ means $\lim_{s \to 0^{+}} r(s) = 0$.  We will often abbreviate a situation such as $f(s) - g(s) = o(a(s))$ as $f(s) = g(s) + o(a(s))$.    

\subsection{Boundary Conditions}  Given a function $q : \mathbb{R} \to \mathbb{R}$ and an $e \in S^{d - 1}$, we define $q_{e} : \mathbb{R}^{d} \to \mathbb{R}$ by 
\begin{equation*}
q_{e}(x) = q\left(\langle x, e \rangle\right).
\end{equation*}
Additionally, if $\epsilon > 0$, $q^{\epsilon}_{e} : \mathbb{R}^{d} \to \mathbb{R}$ is the function defined by 
\begin{equation*}
q^{\epsilon}_{e}(x) = q_{e}(\epsilon^{-1} x) = q \left(\frac{\langle x, e \rangle}{\epsilon} \right).
\end{equation*}

\section{Finite-Volume Surface Tension} \label{S: thermodynamic limit}

First, as in the introduction, we fix a smooth $q : \mathbb{R} \to [-1,1]$ satisfying \eqref{E: BC}.
As we explained already, $q$ serves the role of a boundary condition.  We will see the exact choice of $q$ is immaterial.

\subsection{Finite-volume surface tension}  We begin by introducing the finite-volume surface tension and its basic properties.  
\begin{definition}  The finite-volume surface tension is the random function $\tilde{\Phi}^{\omega} : S^{d - 1} \times \mathbb{R}^{d} \times \mathcal{U}_{0}^{d} \to (0,\infty)$ given by 
\begin{equation*}
\tilde{\Phi}^{\omega}(e,x_{0},A) = \inf \left\{ \mathcal{F}^{\omega}(u; A) \, \mid \, u \in H^{1}(A; [ -1,1]), \, \, u - T_{x_{0}}q_{e} \in H^{1}_{0}(A) \right\}.
\end{equation*}
\end{definition}

Note that the infimum in the definition is, in fact, achieved, as a consequence of the lower semi-continuity of $\mathcal{F}^{\omega}(\cdot ; A)$.  Moreover, since $H^{1}_{0}(A)$ is separable, we readily deduce that $\tilde{\Phi}^{\omega}(e,x_{0},A)$ is $\mathscr{B}$-measurable independently of the choice of $e$, $x_{0}$, and $A$.

Before proceeding further, we record a useful observation.  Here and henceforth, we define the constant $C_{\Lambda} > 0$ by
\begin{equation*}
C_{\Lambda} =  \int_{-\infty}^{\infty} \left(\frac{\Lambda q'(s)^{2}}{2} + W(q(s)) \right) \, ds.
\end{equation*}

\begin{prop} \label{P: basic_upper_bound}  If $e \in S^{d - 1}$, $x_{0} \in \mathbb{R}^{d}$, $A \in \mathcal{U}^{d - 1}_{0}$, and $I \in \mathcal{U}_{0}^{1}$, then
\begin{equation*}
0 \leq \tilde{\Phi}^{\omega}(e,x_{0},A \oplus_{e} I) \leq C_{\Lambda} \mathcal{L}^{d - 1}(A).
\end{equation*}\end{prop}  

\begin{proof}  Using $T_{x_{0}}q$ itself as a candidate, we obtain an upper bound:
\begin{equation*}
0 \leq \tilde{\Phi}^{\omega}(e,x_{0},A \oplus_{e} I) \leq \mathcal{F}^{\omega}(T_{x_{0}}q;A \oplus_{e} I).
\end{equation*}  
It only remains to estimate $\mathcal{F}^{\omega}_{\epsilon}(T_{x_{0}}q;A \oplus_{e} I)$.  Since $\varphi^{\omega}(x,p) \leq \sqrt{\Lambda} \|p\|$ and $q_{e}$ only varies in the $e$ direction, Fubini's Theorem readily implies
\begin{equation*}
\mathcal{F}^{\omega}(T_{x_{0}}q_{e};A \oplus_{e} I) \leq C_{\Lambda} \mathcal{L}^{d - 1}(A).\end{equation*} 
\end{proof}  

%
%

%

\subsection{The case of cubes centered at the origin}  As we mentioned in the introduction, to start with it is convenient to restrict $\tilde{\Phi}^{\omega}$ to $x_{0} = 0$.  The analysis of the resulting process, which we precisely define next, takes up most of the remainder of this section and the one that follows.  

\begin{definition}  The centered finite-volume surface tension is the random function $\tilde{\varphi}^{\omega} : S^{d - 1} \times \mathcal{U}^{d - 1}_{0} \times (0,\infty) \to (0,\infty)$ given by 
\begin{equation*}
\tilde{\varphi}^{\omega}(e,A, h) = \tilde{\Phi}^{\omega}(e,0,A \oplus_{e} (-h,h)).
\end{equation*}
\end{definition}  

After chasing the definitions, the reader will readily verify that the effect of setting $x_{0} = 0$ is to restrict ourselves to the case when the interface $\{T_{x_{0}}q = 0\}$ equals the hyperplane $\langle e \rangle^{\perp}$.  

The next proposition gives the essential properties of the centered surface tension.  In particular, in the language of \cite{nonlinear stochastic homogenization}, $\tilde{\varphi}^{\omega}(e,\cdot,h)$ is a sub-additive process.

\begin{prop} \label{P: important_properties} For each fixed $e \in S^{d - 1}$ and $h > 0$, the function $A \mapsto \tilde{\varphi}^{\omega}(e,A, h)$ satisfies:
\begin{itemize}
\item[(i)] If $A,A_{1},\dots,A_{N} \in \mathcal{U}_{0}^{d - 1}$, $\{A_{1},\dots,A_{N}\}$ is pairwise disjoint, $\bigcup_{i = 1}^{N} A_{i} \subseteq A$, and $\mathcal{L}^{d - 1}(A \setminus \bigcup_{i = 1}^{N} A_{i}) = 0$, then 
\begin{equation} \label{E: subadditive_property}
\tilde{\varphi}^{\omega}(e,A,h) \leq \sum_{i = 1}^{N} \tilde{\varphi}^{\omega}(e,A_{i},h).
\end{equation}
\item[(ii)] $\tilde{\varphi}^{\omega}(e,A,h)$ is uniformly bounded in the following sense:
\begin{equation} \label{E: equibounded}
0 \leq \tilde{\varphi}^{\omega}(e,A,h) \leq C_{\Lambda} \mathcal{L}^{d - 1}(A) \quad \text{if} \, \, A \in \mathcal{U}_{0}^{d - 1}.
\end{equation}
\item[(iii)] For each $A \in \mathcal{U}_{0}^{d -1}$ and each $x \in \langle e \rangle^{\perp}$, the following equation holds: \begin{equation*}
\tilde{\varphi}^{\tau_{x}\omega}(e,A,h) = \tilde{\varphi}^{\omega}(e,A + O_{e}^{-1}(x),h).
\end{equation*}
\end{itemize}\end{prop} 

\begin{proof}  First, observe that (ii) follows directly from Proposition \ref{P: basic_upper_bound}, and (iii) is an immediate consequence of the definitions of $\tilde{\varphi}^{\omega}(e,A,h)$, $\{\tau_{x}\}_{x \in \mathbb{R}^{d}}$, and $\varphi^{\omega}$.  

Next, we prove (i).  Suppose $A,A_{1},\dots,A_{N} \in \mathcal{U}_{0}^{d - 1}$ are given and satisfy the assumptions.  For each $i \in \{1,2,\dots,N\}$, pick $u_{i} \in H^{1}(A_{i} \oplus_{e} (-h,h); [-1,1])$ such that 
\begin{itemize}
\item[(i)] $\mathcal{F}^{\omega}(u_{i}, A_{i} \oplus_{e} (-h,h)) = \tilde{\varphi}^{a}(e,A_{i},h)$ 
\item[(ii)] $u_{i} - q_{e} \in H^{1}_{0}(A_{i} \oplus_{e} (-h,h))$
\end{itemize} 
Define a new function $u : A \oplus_{e} (-h,h) \to [-1,1]$ by 
\begin{equation*}
u(x) = \left\{ \begin{array}{r l}
			u_{i}(x), & x \in A_{i} \oplus_{e} (-h,h) \\
			0, & \text{otherwise}
		\end{array} \right.
\end{equation*}
By the choice of boundary conditions, $u - q_{e} \in H^{1}_{0}(A \oplus_{e} (-h,h))$.  Moreover, 
\begin{equation*}
\tilde{\varphi}^{\omega}(e,A,h) \leq \mathcal{F}^{\omega}(u,A \oplus_{e} (-h,h)) = \sum_{i = 1}^{N} \mathcal{F}^{\omega}(u_{i}, A_{i} \oplus_{e} (-h,h)) = \sum_{i = 1}^{N} \tilde{\varphi}^{\omega}(e,A_{i},h).
\end{equation*}
This establishes (i).   \end{proof}   

We now use the sub-additive ergodic theorem to average out the variations in the medium in directions perpendicular to $e$:

\begin{prop} \label{P: finite horizon limit}  For each $e \in S^{d - 1}$ and $h > 0$, there is a $\Sigma_{e}$-measurable random variable $\tilde{\varphi}^{\omega}(e,h)$ and an event $\Omega_{h} \in \Sigma_{e}$ satisfying $\mathbb{P}(\Omega_{h}) = 1$ such that
\begin{equation*}
\tilde{\varphi}^{\omega}(e,h) = \lim_{R \to \infty} R^{1 - d} \tilde{\varphi}^{\omega}(e,Q(0,R),h) \quad \text{if} \, \, \omega \in \Omega_{h}.
\end{equation*}
 \end{prop}   
 
 \begin{proof}  This is a direct application of Theorem \ref{T: sub-additive_ergodic_theorem} in Appendix \ref{A: ergodic theory}.   
 \end{proof}

%
%

 \subsection{Infinite horizon limit}  We now study the limit $h \to \infty$ of $\tilde{\varphi}^{\omega}(e,A,h)$.  We begin by defining what we call the infinite-horizon surface tension:
 
 \begin{definition}  The infinite-horizon surface tension is the random function $\tilde{\varphi}_{\infty}^{\omega} : S^{d -1} \times \mathcal{U}_{0}^{d - 1} \to (0,\infty)$ given by 
 \begin{equation*}
 \tilde{\varphi}^{\omega}_{\infty}(e,A) = \min \left\{ \mathcal{F}^{\omega}(u; A \oplus_{e} \mathbb{R}) \, \mid \, -1 \leq u \leq 1, \, \, u = q_{e} \, \, \text{on} \, \, \partial A \oplus_{e} \mathbb{R} \right\}.
 \end{equation*}
 \end{definition}  
 
In the main result of this section, we prove that $\lim_{h \to \infty} \tilde{\varphi}^{\omega}(e,A,h) = \tilde{\varphi}^{\omega}_{\infty}(e,A)$, consistent with the way $\tilde{\varphi}^{\omega}_{\infty}(e,A)$ was defined in the introduction.

Before we proceed further, we remark that $\tilde{\varphi}^{\omega}_{\infty}$ satisfies its own version of Proposition \ref{P: important_properties}.  Therefore, the following analogue of Proposition \ref{P: finite horizon limit} holds:

\begin{prop} \label{P: ergodic_infinity}  For each $e \in S^{d - 1}$, there is a $\Sigma_{e}$-measurable random variable $\tilde{\varphi}^{\omega}_{\infty}(e)$ and an event $\Omega_{\infty} \in \Sigma_{e}$ satisfying $\mathbb{P}(\Omega_{\infty}) = 1$ such that
\begin{equation} \label{E: key}
\tilde{\varphi}^{\omega}_{\infty}(e) = \lim_{R \to \infty} R^{1 - d} \tilde{\varphi}^{\omega}_{\infty}(e,Q(0,R)) \quad \text{if} \, \, \omega \in \Omega_{\infty}.
\end{equation}
\end{prop}

It only remains to analyze the infinite horizon limit $h \to \infty$.  We start by observing that $h \mapsto \tilde{\varphi}^{\omega}(e,A,h)$ is \emph{almost} non-increasing.
 
\begin{prop} \label{P: almost_decreasing_prop}  Fix $e \in S^{d - 1}$ and $A \in \mathcal{U}^{d - 1}_{0}$.  If $h_{1} > h _{2}$, then 
\begin{equation} \label{E: almost_decreasing}
\tilde{\varphi}^{\omega}(e,A,h_{1}) \leq \tilde{\varphi}^{\omega}(e,A,h_{2}) + \mathcal{L}^{d - 1}(A) e(h_{2}),
\end{equation} 
where 
\begin{equation*}
e(h) = \int_{\{|s| > h\}} \left(\frac{\Lambda q'(s)^{2}}{2} + W(q(s)) \right) \, ds.
\end{equation*}
In particular, for each $e \in S^{d - 1}$, $R > 0$, and $\omega \in \Omega$, $\lim_{h \to \infty} \tilde{\varphi}^{\omega}(e,Q^{e}(0,R),h)$ exists.  \end{prop}

We will see that the proof of Proposition \ref{P: almost_decreasing_prop} does not use any properties of the probability space $(\Omega,\mathscr{B},\mathbb{P})$ or the action $\tau$.

\begin{proof}  Fix $\epsilon > 0$.  Choose a $u : A \oplus_{e} (-h_{2},h_{2}) \to [-1,1]$ such that $u - q_{e} \in H_{0}^{1}(A \oplus_{e} (-h_{2},h_{2}))$ and 
\begin{equation*}
\mathcal{F}^{\omega}(u; A \oplus_{e} (-h_{2},h_{2})) = \tilde{\varphi}^{\omega}(e,A, h_{2}).
\end{equation*}
Define $\tilde{u} : A \oplus (-h_{1},h_{1}) \to [-1,1]$ by 
\begin{equation*}
\tilde{u}(x) = \left\{ \begin{array}{r l}
				u(x), & x \in A \oplus_{e} (-h_{2},h_{2}) \\
				q_{e}(x), & \text{otherwise}
				\end{array} \right.
\end{equation*}
Then $\tilde{u} - q_{e} \in H^{1}_{0}(A \oplus_{e} (-h_{1},h_{1}))$ and
\begin{align*}
\tilde{\varphi}^{\omega}(e,A,h_{1}) &\leq \mathcal{F}^{\omega}(\tilde{u},A \oplus_{e} (-h_{1},h_{1})) \\
	&= \mathcal{F}^{\omega}(u,A \oplus (-h_{2},h_{2})) + \mathcal{L}^{d - 1}(A) \int_{\{|t| \in (h_{2},h_{1})\}} \left(\Lambda q'(t)^{2} + W(q(t)) \right) \, dt \\
	&\leq \tilde{\varphi}^{\omega}(e,A,h_{2}) + \mathcal{L}^{d - 1}(A) e(h_{2}).
\end{align*}
Thus, we obtain \eqref{E: almost_decreasing}.

Finally, sending $h_{1} \to \infty$ with $h_{2}$ fixed and then sending $h_{2} \to \infty$, we find
\begin{equation*}
\limsup_{h_{1} \to \infty} \tilde{\varphi}^{\omega}(e,A,h_{1}) \leq \liminf_{h_{2} \to \infty} \left(\tilde{\varphi}^{\omega}(e,A,h_{2}) + \mathcal{L}^{d - 1}(A) e(h_{2})\right) = \liminf_{h_{2} \to \infty} \tilde{\varphi}^{\omega}(e,A,h_{2}).  
\end{equation*}
This proves $\lim_{h \to \infty} \tilde{\varphi}^{\omega}(e,A,h)$ exists. \end{proof}

Finally, we prove the result that was promised at the beginning of this sub-section:

\begin{prop} \label{P: nice characterization}  For each $e \in S^{d - 1}$ and $A \in \mathcal{U}_{0}^{d - 1}$, we have
\begin{equation} \label{E: infinite-horizon}
\tilde{\varphi}^{\omega}_{\infty}(e,A) = \lim_{h \to \infty} \tilde{\varphi}^{\omega}(e,A,h).
\end{equation}
\end{prop}  

\begin{proof}  If $h > 0$ and $u \in H^{1}(A \oplus (-h,h); [-1,1])$ equals $q_{e}$ on $\partial (A \oplus_{e} (-h,h))$, then the function $\tilde{u} \in H^{1}_{\text{loc}}(A \oplus_{e} \mathbb{R})$ given by 
\begin{equation*}
\tilde{u}(x) = \left\{ \begin{array}{r l}
					u(x), & |\langle x, e \rangle| \leq h \\
					q_{e}(x), & \text{otherwise}
				\end{array} \right.
\end{equation*} 
satisfies $\tilde{u} - q_{e} \in H^{1}_{0}(A \oplus_{e} \mathbb{R})$.  Thus,
\begin{equation*}
\tilde{\varphi}^{\omega}_{\infty}(e,A) \leq \mathcal{F}^{\omega}(\tilde{u}; A \oplus_{e} \mathbb{R}) \leq \mathcal{F}^{\omega}(u; A \oplus_{e} (-h,h)) + \mathcal{L}^{d -1}(A) e(h).
\end{equation*}
Since $u$ was arbitrary, we deduce that $\tilde{\varphi}^{\omega}_{\infty}(e,A) \leq \tilde{\varphi}^{\omega}(e,A,h) + \mathcal{L}^{d - 1}(A) e(h)$.  Sending $h \to \infty$, we conclude $ \tilde{\varphi}^{a}_{\infty}(e,A) \leq \lim_{h \to \infty} \tilde{\varphi}^{\omega}(e,A,h)$.    

To obtain the complementary inequality, let $u \in H^{1}_{\text{loc}}(A \oplus_{e} \mathbb{R}; [-1,1])$ be any function attaining the minimum in \eqref{E: infinite-horizon}.  We will show that it is possible to appropriately truncate $u$ without changing its energy too much.  

First, observe that for each $\delta > 0$,
\begin{align*}
\lim_{R \to \infty} \mathcal{L}^{d}(\{x \in A \oplus_{e} [R,+\infty) \, \mid \, |u(x) - 1| > \delta\}) &= 0 \\
\lim_{R \to \infty} \mathcal{L}^{d}(\{x \in A \oplus_{e} (-\infty,-R] \, \mid \, |u(x) + 1| > \delta\} &= 0.
\end{align*}
This is a consequence of the Poincar\'{e} inequality, which in this setting states 
\begin{equation*}
\|u - q_{e}\|_{L^{2}(A \oplus_{e} \mathbb{R})} \leq \|Du - Dq_{e}\|_{L^{2}(A \oplus_{e} \mathbb{R})}.
\end{equation*}

For each $n \in \mathbb{N}$, fix a smooth function $f_{n} : \mathbb{R} \to [0,1]$ such that $f_{n} \equiv 1$ in $[-n,n]$, $f_{n} \equiv 0$ in $\mathbb{R} \setminus [-(n + 1),n+1]$, and $|f_{n}'| \leq 2$.  Let $u_{n} \in H^{1}_{\text{loc}}(A \oplus_{e} \mathbb{R})$ be the function defined by 
\begin{equation*}
u_{n}(x) = f_{n}(\langle x, e \rangle) u(x) + (1 - f_{n}(\langle x, e \rangle)) q_{e}(x).
\end{equation*} 
We readily obtain the following bounds on the energy of $u_{n}$:
\begin{align*}
\mathcal{F}^{\omega}(u_{n}; A \oplus_{e} (-(n + 1),n + 1)) &\leq \mathcal{F}^{\omega}(u; A \oplus_{e} (-n,n))  \\
		&\quad \quad + \frac{\Lambda}{2} \int_{A \oplus_{e} \{n < |s| < n + 1\}} |Du(x)|^{2} \, dx \\
		&\quad \quad + \frac{\Lambda}{2} \mathcal{L}^{d -1}(A) \int_{\{n \leq |s| \leq n + 1\}} q'(s)^{2} \, ds \\
		&\quad \quad + 2 \int_{A \oplus_{e} \{n < | s | < n + 1\}} |u(x) - q_{e}(x)|^{2} \, dx \\
		&\quad \quad + \int_{A \oplus_{e} \{n < |s| < n + 1\}} W(u_{n}(x)) \, dx
\end{align*}   
Since $u, q_{e} \to \pm 1$ in measure as $\langle x, e \rangle \to \pm \infty$, we find  
\begin{equation*}
\mathcal{F}^{\omega}(u_{n}; A \oplus_{e} (-(n + 1),n + 1)) \leq \tilde{\varphi}^{\omega}_{\infty}(e,A) + o(1)
\end{equation*}
as $n \to \infty$.  Therefore, since $u_{n} = q_{e}$ on $A \oplus_{e} \{-(n + 1),n+1\}$, we conclude
\begin{equation*}
\lim_{h \to \infty} \tilde{\varphi}^{\omega}(e,A,h) \leq \lim_{n \to \infty} \mathcal{F}^{\omega}(u_{n}; A \oplus_{e} (-(n + 1),n + 1)) \leq \tilde{\varphi}^{\omega}_{\infty}(e,A).
\end{equation*}  
\end{proof}

\section{Infinite-volume surface tension}  \label{S: surface tension}

We now identify the infinite-volume surface tension $\tilde{\varphi}$.  To start with, it's convenient to define this as a random function.  We show the surface tension is deterministic almost surely using translation invariance, and soon thereafter we prove the thermodynamic limit.  

\begin{definition}  The infinite-volume surface tension is the random function $\tilde{\varphi}^{\omega} : S^{d - 1} \to [0,\infty)$ given by
\begin{equation} \label{E: infinite volume}
\tilde{\varphi}^{\omega}(e) = \liminf_{R \to \infty} R^{1 - d} \tilde{\varphi}^{\omega}(e,Q(0,R),R).
\end{equation}\end{definition}  

Our first observation is this quantity is translationally invariant:

\begin{theorem} \label{T: translation_invariance} For each $e \in S^{d - 1}$, $\tilde{\varphi}^{\omega}(e)$ is $\Sigma$-measurable. \end{theorem}  

An immediate consequence of the theorem and ergodicity is $\tilde{\varphi}^{\omega}(e)$ is constant almost surely.  We record this in the following corollary:

\begin{corollary}  There is a unique $\tilde{\varphi}(e) \geq 0$ depending only on $\mathbb{P}$ such that $\tilde{\varphi}^{\omega}(e) = \tilde{\varphi}(e)$ almost surely.  \end{corollary}  

Theorem \ref{T: translation_invariance} will be proved in two steps.  In the first, we note that $\tilde{\varphi}^{\omega}(e)$ is invariant under translations in directions perpendicular to $e$.  This step follows from what we already proved in Section \ref{S: thermodynamic limit}, especially Propositions \ref{P: important_properties} and \ref{P: almost_decreasing_prop}.  In the second step, we show that $\tilde{\varphi}^{\omega}(e)$ is invariant under the action of translations in the $e$ direction.  The proof of this is very similar to that of Proposition \ref{P: nice characterization}. 

In the remainder of the section, we study the thermodynamic limit and prove Theorem \ref{T: thermodynamic_limit}.

\subsection{Tangential directions}  To simplify the proof of translation invariance, we first observe that $\tilde{\varphi}^{\omega}(e)$ is invariant under translations in directions perpendicular to $e$:

\begin{prop} \label{P: tangential invariance} $\tilde{\varphi}^{\omega}(e)$ is $\Sigma_{e}$-measurable.
\end{prop}  

\begin{proof}  We begin by showing that $\tilde{\varphi}^{\tau_{x} \omega}(e) = \tilde{\varphi}^{\omega}(e)$ if $x \in \langle e \rangle^{\perp}$.  At the end of the proof, we show $\tilde{\varphi}^{\omega}(e)$ is $\mathscr{B}$-measurable.  

Supppose $x \in \langle e \rangle^{\perp}$.  We will show that $\tilde{\varphi}^{\tau_{x} \omega}(e) \leq \tilde{\varphi}^{\omega}(e)$.  Let $y = O_{e}^{-1}(x)$.  If $R > 0$, then Proposition \ref{P: important_properties} implies
\begin{equation} \label{E: basic 1}
\tilde{\varphi}^{\tau_{x}\omega}(e,Q(0,R + |y|_{\infty}), R + |y|_{\infty}) = \tilde{\varphi}^{\omega}(e,Q(y,R + |y|_{\infty}),R + |y|_{\infty}).
\end{equation}
Since $Q(y,R + |y|_{\infty}) \supseteq Q(0,R)$, we use \eqref{E: subadditive_property} and \eqref{E: equibounded} from Proposition \ref{P: important_properties} to find
\begin{equation} \label{E: basic 2}
\tilde{\varphi}^{\omega}(e,Q(y,R + |y|_{\infty}),R + |y|_{\infty}) \leq \tilde{\varphi}^{\omega}(e,Q(0,R),R + |y|_{\infty}) + C_{\Lambda} R^{d - 1} o(1).
\end{equation}
as $R \to \infty$.  
Finally, appealing to \eqref{E: almost_decreasing} from Proposition \ref{P: almost_decreasing_prop} yields
\begin{equation} \label{E: basic 3}
\tilde{\varphi}^{\omega}(e,Q(y,R + |y|_{\infty}),R + |y|_{\infty}) \leq \tilde{\varphi}^{\omega}(e,Q(0,R),R) + R^{d - 1}(e(R) + C_{\Lambda} o(1)).
\end{equation}
Combining \eqref{E: basic 1}, \eqref{E: basic 2}, and \eqref{E: basic 3}, dividing by $R^{d - 1}$, and sending $R \to \infty$, we obtain $\tilde{\varphi}^{\tau_{x}\omega}(e) \leq \tilde{\varphi}^{\omega}(e)$.  

Replacing $x$ with $-x$ and $\omega$ with $\tau_{x}\omega$ yields $\tilde{\varphi}^{\omega}(e) \leq \tilde{\varphi}^{\tau_{x}\omega}(e)$.  Therefore, $\tilde{\varphi}^{\tau_{x} \omega}(e) = \tilde{\varphi}^{\omega}(e)$.  

It only remains to show that $\tilde{\varphi}^{\omega}(e)$ is $\mathscr{B}$-measurable.  To do so, it suffices to verify the following identity:
\begin{equation*}
\tilde{\varphi}^{\omega}(e) = \liminf_{\mathbb{N} \ni N \to \infty} N^{1 - d} \tilde{\varphi}^{\omega}(e,Q(0,N),N).
\end{equation*}
This can be derived using arguments very similar to those presented earlier in the proof.  We omit the details.   
\end{proof}  
 
 \subsection{The normal direction} Since any $x \in \mathbb{R}^{d}$ can be decomposed as $x = \tilde{x} + t e$ for some $t \in \mathbb{R}$ and $\tilde{x} \in \langle e \rangle^{\perp}$, we obtain the following decomposition of $\tau_{x}$: 
 \begin{equation*}
 \tau_{x} = \tau_{\tilde{x}} \circ \tau_{t e} = \tau_{te} \circ \tau_{\tilde{x}}.
 \end{equation*}
 Thus, since $\tilde{\varphi}^{\omega}(e)$ is $\Sigma_{e}$-measurable, Theorem \ref{T: translation_invariance} is proved as soon as we establish that $\tilde{\varphi}^{\omega}(e)$ is invariant under $(\tau_{te})_{t \in \mathbb{R}}$.  We prove this next using the fundamental estimate of $\Gamma$-convergence.
 
\begin{prop} \label{P: vertical_invariance}  If $\omega \in \Omega$ and $t \in \mathbb{R}$, then 
\begin{equation*}
\tilde{\varphi}^{\omega}(e) = \tilde{\varphi}^{\tau_{te}\omega}(e).
\end{equation*}
\end{prop}

\begin{proof}  Fix $\omega \in \Omega$ and $t \in \mathbb{R}$.  We begin by showing that the following inequality holds:
\begin{equation*}
\tilde{\varphi}^{\omega}(e) \leq \tilde{\varphi}^{\tau_{te}\omega}(e)
\end{equation*}
 It is convenient to introduce a free parameter $\alpha \in (0,1)$.  In the final step of the proof, we will send $\alpha \to 1^{-}$.  

For each $R > 0$, let $\tilde{v}_{R} \in H^{1}(Q^{e}(0,\alpha R); [-1,1])$ be a minimizer of the functional $\mathcal{F}^{\tau_{te}\omega}(\cdot; Q^{e}(0,\alpha R))$ subject to the boundary conditions $\tilde{v}_{R} = q_{e}$, that is, 
\begin{equation*}
\tilde{v}_{R} - q_{e} \in H^{1}_{0}(Q^{e}(0, \alpha R)), \quad \mathcal{F}^{\tau_{te}\omega}(\tilde{v}_{R}; Q^{e}(0, \alpha R)) = \tilde{\varphi}^{\tau_{te}\omega}(e,Q(0,\alpha R), \alpha R).
\end{equation*}  

In order to compare $\tilde{\varphi}^{\omega}(e,Q(0,R),R)$ and $\tilde{\varphi}^{\tau_{te}\omega}(e,Q(0,\alpha R),\alpha R)$, we shift perspectives by defining the function $v_{R}$ in $Q^{e}(te,\alpha R)$ by 
\begin{equation*}
v_{R}(x) = \tilde{v}_{R}(x - te).
\end{equation*}  
Notice that $v_{R}$ minimizes $\mathcal{F}^{\omega}(\cdot; Q^{e}(te,\alpha R))$ with the boundary condition $T_{te}q_{e}$.  We extend $v_{R}$ to $\mathbb{R}^{d}$ by setting $v_{R}(x) = T_{te}q_{e}(x)$ if $x \in \mathbb{R}^{d} \setminus Q^{e}(te,\alpha R)$.

Finally, before we apply the fundamental estimate, it is convenient to move to macroscopic coordinates.  We let $\epsilon = R^{-1}$ and define a new function $v^{\epsilon} : \mathbb{R}^{d} \to [-1,1]$ by the following rule:
 \begin{equation*}
 v^{\epsilon}(x) = v_{R}\left(\frac{x}{\epsilon}\right).
 \end{equation*}        
The effect of the definition is this: $v^{\epsilon}$ minimizes $\mathcal{F}_{\epsilon}^{\omega}(\cdot ; Q^{e}(\epsilon t e, \alpha))$ subject to the boundary condition $T_{\epsilon te}(q^{\epsilon}_{e})$ on $\partial Q^{e}(\epsilon te, \alpha)$.

We will now apply the fundamental estimate with the $\epsilon$-independent open sets $U$, $U'$, and $V$ defined as follows:
\begin{equation*}
U = Q^{e}(0,\alpha), \, \, U' = Q^{e} \left(0,1\right), \, \, V = Q^{e}(0,1) \setminus Q^{e}(0, \alpha).
\end{equation*}
We will work with the functions $(v_{\epsilon})_{\epsilon > 0}$ and $(q_{e}^{\epsilon})_{\epsilon > 0}$.  Observe that 
\begin{equation*}
\lim_{\epsilon \to 0^{+}} \left( \|v^{\epsilon} - q^{\epsilon}_{e}\|_{L^{1}(V)} + \|q^{\epsilon}_{e} - \chi_{e}\|_{L^{1}(V)} \right) = 0.
\end{equation*}  
Thus, by the fundamental estimate, there is a function $\tilde{e} : (0,\infty) \to (0,\infty)$ such that $\lim_{\epsilon \to 0^{+}} \tilde{e}(\epsilon) = 0$ and a family of cut-off functions $(\psi_{\epsilon})_{\epsilon > 0} \subseteq C^{\infty}_{c}(U';[0,1])$ satisfying $\psi_{\epsilon} \equiv 1$ in $U$ such that
\begin{equation*}
\mathcal{F}_{\epsilon}^{\omega}(\psi_{\epsilon} v_{\epsilon} + (1 - \psi_{\epsilon})q_{e}^{\epsilon} ; U \cup V) \leq \mathcal{F}_{\epsilon}^{\omega}(v_{\epsilon}; U') + \mathcal{F}_{\epsilon}^{\omega}(q_{e}^{\epsilon}; V) + \tilde{e}(\epsilon).
\end{equation*}

Now observe that we can make the following simplifications:
\begin{align*}
\mathcal{F}^{\omega}_{\epsilon}(v_{\epsilon};U') &= R^{1 - d} \mathcal{F}^{\tau_{te}\omega}(\tilde{v}_{R}; Q^{e}(0,\alpha R)) + R^{1 -d} \mathcal{F}^{\omega}(v_{R}; RU' \setminus Q^{e}(te, \alpha R)) \\
		&= R^{1 - d} \tilde{\varphi}^{\tau_{te}a}(0,Q(0,\alpha R),\alpha R) + \omega(\alpha) + \eta(\alpha,R),
\end{align*} 
where 
\begin{align*}
\omega(\alpha) &\leq (1 - d) (1 - \alpha) C_{\Lambda} \\
\eta(\alpha,R) &\leq \alpha^{d -1} \int_{\{|s| \geq \alpha R\}} \left(\Lambda q'(s)^{2} + W(q(s)) \right) \, ds
\end{align*}
Similarly, $\mathcal{F}^{\omega}_{\epsilon}(q_{e}^{\epsilon}; V) = \nu(\alpha) + \gamma(\alpha, R)$, where
\begin{align*}
\nu(\alpha) &\leq  (1 - d) \left(1 - \alpha\right) C_{\Lambda} \\
\gamma(\alpha,R) &\leq \alpha^{d - 1} \int_{\{ |s| \geq \alpha R - |t|\}} \left(\Lambda q'(s)^{2} + W(q(s)) \right) \, ds
\end{align*}  
Sending $R \to \infty$ and observing that $\lim_{R \to \infty} (\eta(\alpha,R) + \gamma(\alpha,R)) = 0$, we obtain
\begin{align*}
\liminf_{R \to \infty} R^{1 -d} \tilde{\varphi}^{\omega}(e,Q(0,R),R) &\leq \liminf_{\epsilon \to 0^{+}} \mathcal{F}_{\epsilon}^{\omega}(\psi_{\epsilon} v_{\epsilon} + (1 - \psi_{\epsilon})q_{e}^{\epsilon} ; U \cup V) \\
	&\leq \alpha^{d -1} \liminf_{T \to \infty} T^{1 -d} \tilde{\varphi}^{\tau_{te}\omega}(0,Q(0,T),T) \\
	&\qquad+ \omega(\alpha) + \nu(\alpha).
\end{align*}
Since $\lim_{\alpha \to 1^{-}} \left(\omega_{R}(\alpha) + \nu(\alpha)\right) = 0$ and $\alpha \in (0,1)$ was arbitrary, we conclude
\begin{equation*}
\tilde{\varphi}^{\omega}(e) = \liminf_{R \to \infty} R^{1 -d} \tilde{\varphi}^{\omega}(e,Q(0,R),R) \leq \liminf_{R \to \infty} R^{1 -d} \tilde{\varphi}^{\tau_{te}\omega}(0,Q(0,R),R) = \tilde{\varphi}^{\tau_{te}\omega}(e).
\end{equation*}

Replacing $t$ with $-t$ and $\omega$ with $\tau_{te}\omega$, we find $\tilde{\varphi}^{\tau_{te}\omega}(e) \leq \tilde{\varphi}^{\omega}(e)$.
Therefore, $\tilde{\varphi}^{\tau_{te}\omega}(e) = \tilde{\varphi}^{\omega}(e)$.   
\end{proof}  

\subsection{Thermodynamic Limit}  Now we tie all the pieces together.  While we know that $\tilde{\varphi}^{\omega}(e)$ is a constant almost surely, we do not know that the limit inferior in its definition coincides with the limit superior.  We need to prove this occurs with probability one to proceed with the proof of $\Gamma$-convergence.  This is where Theorem \ref{T: thermodynamic_limit} comes in.  

In what follows, the following lemma will be helpful:

\begin{lemma} \label{L: helpful_ergodic} For each $n, m \in \mathbb{N}$, let $\mathbb{E}(\tilde{\varphi}^{\omega}(e,Q(0,2^{n}),m) \mid \Sigma_{e})$ and $\mathbb{E}(\tilde{\varphi}^{\omega}_{\infty}(e,Q(0,2^{n})) \mid \Sigma_{e})$ denote fixed representatives of the conditional expectations with respect to $\Sigma_{e}$ of $\tilde{\varphi}^{\omega}(e,Q(0,2^{n}),m)$ and $\tilde{\varphi}^{\omega}_{\infty}(e,Q(0,2^{n}))$, respectively.  There are $\Sigma_{e}$-measurable events $\Omega_{\leq}$, $\Omega_{\ell}$, and $\Omega_{\text{horizon}}$ satisfying the following properties:
\begin{itemize}
\item[(i)] $\mathbb{P}(\Omega_{\ell}) = \mathbb{P}(\Omega_{\leq}) =  \mathbb{P}(\Omega_{\text{horizon}}) = 1$
\item[(ii)] If $\omega \in \Omega_{\leq}$ and $n,m \in \mathbb{N}$, then
\begin{equation*}
2^{(n + 1)(1 - d)} \mathbb{E}(\tilde{\varphi}^{\omega}(e,Q(0,2^{n + 1}),m) \mid \Sigma_{e}) \leq 2^{n(1 - d)} \mathbb{E}(\tilde{\varphi}^{\omega}(e,Q(0,2^{n}),m) \mid \Sigma_{e}).
\end{equation*}
\item[(iii)] If $\omega \in \Omega_{\ell}$, then
\begin{align*}
\tilde{\varphi}^{\omega}(e,m) &= \lim_{n \to \infty} 2^{n(1 - d)} \mathbb{E}(\tilde{\varphi}^{\omega}(e,Q(0,2^{n}),m) \mid \Sigma_{e}) \\
\tilde{\varphi}^{\omega}(e) &= \lim_{n \to \infty} 2^{n(1 - d)} \mathbb{E}(\tilde{\varphi}^{\omega}(e,Q(0,2^{n})) \mid \Sigma_{e}).
\end{align*}
\item[(iv)] If $\omega \in \Omega_{\text{horizon}}$, then
\begin{align*}
\mathbb{E}(\tilde{\varphi}^{\omega}_{\infty}(e,Q(0,2^{n})) \mid \Sigma_{e}) = \lim_{m \to \infty} \mathbb{E}(\tilde{\varphi}^{\omega}(e,Q(0,2^{n}),m).
\end{align*}
\end{itemize}
\end{lemma}  

The idea of parts (ii) and (iii) is that the $\Sigma_{e}$-conditional expectation preserves, and even improves, the sub-additivity properties of the finite-volume surface tension.  As we will see, conditional expectation allows us to verify one of the crucial inequalities in the proof of Theorem \ref{T: thermodynamic_limit}.

For the most part, the lemma follows from what we have already proven and elementary properties of conditional expectations.  We defer the proof to the end of the sub-section and proceed to the proof of the thermodynamic limit.


\begin{proof}[Proof of Theorem \ref{T: thermodynamic_limit}]  We identify $\tilde{\Omega}_{e}$ using Propositions \ref{P: finite horizon limit} and \ref{P: ergodic_infinity} and Lemma \ref{L: helpful_ergodic}, and then we prove that the following relations hold:
\begin{align}
\lim_{R \to \infty} R^{1 - d} \tilde{\varphi}^{\omega}_{\infty}(e,Q(0,R)) &= \lim_{h \to \infty} \limsup_{R \to \infty} R^{1 - d} \tilde{\varphi}^{\omega}(e,Q(0,R),h) \label{E: sup-inf} \\
		&= \lim_{h \to \infty} \liminf_{R \to \infty} R^{1 - d} \tilde{\varphi}^{\omega}(e,Q(0,R),h) \nonumber \\
\lim_{R \to \infty} R^{1 - d} \tilde{\varphi}_{\infty}^{\omega}(e,Q(0,R)) &= \lim_{R \to \infty} R^{1 - d} \tilde{\varphi}^{\omega}(e,Q(0,R), \kappa R) \label{E: cubes}
\end{align}
Here $\kappa \in (0,\infty)$ is arbitrary.  Notice that once this is done, we set $\kappa = 1$ to conclude that each of these quantities equals $\tilde{\varphi}^{\omega}(e)$, which, in turn, equals $\tilde{\varphi}(e)$ almost surely.  
  
Let $\Omega_{\infty}$ be the event defined in Proposition \ref{P: ergodic_infinity} and let $\Omega_{0} = \bigcap_{n = 1}^{\infty} \Omega_{n}$, where $\{\Omega_{n}\}_{n \in \mathbb{N}}$ are the events defined in Proposition \ref{P: finite horizon limit}.  For each $n, m \in \mathbb{N}$, fix representatives $\mathbb{E}(\tilde{\varphi}^{\omega}(e,Q(0,2^{n}),m) \mid \Sigma_{e})$ and $\mathbb{E}(\tilde{\varphi}^{\omega}(e,Q(0,2^{n})) \mid \Sigma_{e})$ of the $\Sigma_{e}$-conditional expectations of $\tilde{\varphi}^{\omega}(e,Q(0,2^{n}),m)$ and $\tilde{\varphi}^{\omega}_{\infty}(e,Q(0,2^{n}))$, respectively.  Let $\Omega_{\leq}, \Omega_{\ell}, \Omega_{\text{horizon}} \in \Sigma_{e}$ be the sets defined in Lemma \ref{L: helpful_ergodic}.  Finally, define $\tilde{\Omega}_{e}$ by 
\begin{equation*}
\tilde{\Omega}_{e} = \Omega_{\ell} \cap \Omega_{\leq} \cap \Omega_{0} \cap \Omega_{\infty} \cap \{\omega \in \Omega \, \mid \, \tilde{\varphi}^{\omega}(e) = \tilde{\varphi}(e)\}.
\end{equation*}  
As the intersection of $\Sigma_{e}$-measurable events, $\tilde{\Omega}_{e} \in \Sigma_{e}$.  Similarly, $\mathbb{P}(\tilde{\Omega}_{e}) = 1$. 

Assume henceforth that $\omega \in \tilde{\Omega}_{e}$.

To establish \eqref{E: sup-inf}, we first prove 
\begin{equation} \label{E: easy_part_commute}
\lim_{R \to \infty} R^{1 -d} \tilde{\varphi}^{\omega}_{\infty}(e,Q(0,R)) \leq \lim_{h \to \infty} \liminf_{R \to \infty} R^{1 - d} \tilde{\varphi}^{\omega}(e,Q(0,R),h).
\end{equation}
Indeed, by \eqref{E: almost_decreasing} and \eqref{E: infinite-horizon}, 
\begin{equation*}
R^{1 - d} \tilde{\varphi}^{\omega}_{\infty}(e,Q(0,R)) \leq R^{1 - d} \tilde{\varphi}^{\omega}(e,Q(0,R), h) + e(h).
\end{equation*}
Thus, sending $R \to \infty$ and then $h \to \infty$, we obtain \eqref{E: easy_part_commute}.

Next, we prove the inequality complementary to \eqref{E: easy_part_commute}, that is, we show that 
\begin{equation} \label{E: hard_part_commute}
\lim_{h \to \infty} \limsup_{R \to \infty} R^{1 - d} \tilde{\varphi}^{\omega}(e,Q(0,R),h) \leq \lim_{R \to \infty} R^{1 -d} \tilde{\varphi}^{\omega}_{\infty}(e,Q(0,R)).
\end{equation}
We begin by observing that if $h_{1} > h_{2}$, then
\begin{equation*}
 \limsup_{R \to \infty} R^{1 - d} \tilde{\varphi}^{\omega}(e,Q(0,R),h_{1}) \leq \limsup_{R \to \infty} R^{1 - d} \tilde{\varphi}^{\omega}(e,Q(0,R),h_{2}) + e(h_{2}).  
\end{equation*}
Thus, arguing as in Proposition \ref{P: almost_decreasing_prop}, we see that 
	\begin{equation*}
		\lim_{h \to \infty} \limsup_{R \to \infty} R^{1 -d} \tilde{\varphi}^{\omega}(e,Q(0,R),h) \quad \text{exists}
	\end{equation*} 
with no further assumptions on $\omega$.  

Since $\tilde{\Omega}_{e} \subseteq \Omega_{\ell}$, the following equations hold:
\begin{align*}
\lim_{R \to \infty} R^{1 - d} \tilde{\varphi}^{\omega}_{\infty}(e,Q(0,R)) &= \lim_{n \to \infty} 2^{n(1 -d)} \mathbb{E}(\tilde{\varphi}^{\omega}_{\infty}(e,Q(0,2^{n})) \mid \Sigma_{e}) \\
\lim_{h \to \infty} \limsup_{R \to \infty} R^{1 - d} \tilde{\varphi}^{\omega}(e,Q(0,R),h) &= \lim_{m \to \infty} \lim_{n \to \infty} 2^{n(1 - d)} \mathbb{E}(\tilde{\varphi}^{\omega}(e,Q(0,2^{n}),m)) \mid \Sigma_{e}).
\end{align*}  
Thus, to obtain an inequality between the left-hand sides it suffices to complete the easier task of comparing the right-hand sides.

To proceed, we observe that, by definition of $\Omega_{\leq}$,
\begin{equation*}
2^{(n + 1)(1 - d)} \mathbb{E}(\tilde{\varphi}^{\omega}(e,Q(0,2^{n +1}),m) \mid \Sigma_{e}) \leq 2^{n(1 - d)} \mathbb{E}(\tilde{\varphi}^{\omega}(e,Q(0,2^{n}),m) \mid \Sigma_{e}).
\end{equation*}
Consequently, for each fixed $k \in \mathbb{N}$,
\begin{align*}
\lim_{n \to \infty} 2^{n(1 - d)} \mathbb{E}(\tilde{\varphi}^{\omega}(e,Q(0,2^{n}),m) \mid \Sigma_{e}) 
	&\leq 2^{k(1 - d)} \mathbb{E}(\tilde{\varphi}^{\omega}(e,Q(0,2^{k}),m) \mid \Sigma_{e}).
\end{align*}
Sending $m \to \infty$ while $k$ remains fixed, we use the fact that $\tilde{\Omega}_{e} \subseteq \Omega_{\text{horizon}}$ to find
\begin{equation*}
\lim_{m \to \infty} \lim_{n \to \infty} 2^{n(1 - d)} \mathbb{E}(\tilde{\varphi}^{\omega}(e,Q(0,2^{n}),m) \mid \Sigma_{e}) \leq 2^{k(1 - d)} \mathbb{E}(\tilde{\varphi}^{\omega}_{\infty}(e,Q(0,2^{k})) \mid \Sigma_{e}).
\end{equation*}
Taking $k \to \infty$, this becomes
\begin{align*}
\lim_{m \to \infty} \lim_{n \to \infty} 2^{n(1 - d)} \mathbb{E}(\tilde{\varphi}^{\omega}(e,Q(0,2^{n}),m) \mid \Sigma_{e}) &\leq \lim_{k \to \infty} 2^{k(1 - d)} \mathbb{E}(\tilde{\varphi}^{\omega}_{\infty}(e,Q(0,2^{k})) \mid \Sigma_{e}).
\end{align*}
Appealing to the observation in the previous paragraph, we conclude \eqref{E: hard_part_commute} holds.
Now \eqref{E: sup-inf} follows from \eqref{E: easy_part_commute} and \eqref{E: hard_part_commute}.

We proceed to the proof of \eqref{E: cubes}.  Again, we break the equality into two inequalities.  Fix $h > 0$.  Recalling \eqref{E: almost_decreasing} and sending $R \to \infty$, we find
\begin{equation*}
\limsup_{R \to \infty} R^{1 - d} \tilde{\varphi}^{\omega}(e,Q(0,R),\kappa R) \leq \limsup_{R \to \infty} R^{1 - d} \tilde{\varphi}^{\omega}(e,Q(0,R),h) + e(h).
\end{equation*}
Sending $h \to \infty$ and appealing to \eqref{E: sup-inf}, this becomes
\begin{equation*}
\limsup_{R \to \infty} R^{1 - d} \tilde{\varphi}^{\omega}(e,Q(0,R),\kappa R) \leq \lim_{R \to \infty} R^{1 - d} \tilde{\varphi}^{\omega}_{\infty}(e,Q(0,R)).
\end{equation*}

To deduce the opposite inequality, first observe that, by passing to the limit $h \to \infty$ in \eqref{E: almost_decreasing}, we obtain
\begin{equation*}
R^{1 -d} \tilde{\varphi}^{\omega}_{\infty}(e,Q(0,R)) \leq R^{1 - d} \tilde{\varphi}^{\omega}(e,Q(0,R),\kappa R) + e(\kappa R).
\end{equation*}
Thus, we conclude, after sending $R \to \infty$,
\begin{equation*}
\lim_{R \to \infty} R^{1 - d} \tilde{\varphi}^{\omega}_{\infty}(e,Q(0,R)) \leq \liminf_{R \to \infty} R^{1 - d} \tilde{\varphi}^{\omega}(e,Q(0,R),\kappa R).
\end{equation*}
Therefore, $\lim_{R \to \infty} R^{1 - d} \tilde{\varphi}^{\omega}(e,Q(0,R),\kappa R)$ exists and \eqref{E: cubes} holds.  
\end{proof} 

Now that Theorem \ref{T: thermodynamic_limit} is proved, a few remarks are in order:

\begin{remark}  It is not hard to show 
	\begin{equation*}
		\left\{\omega \in \Omega \, \mid \, \lim_{R \to \infty} R^{1 - d} \tilde{\varphi}^{\omega}(e,Q(0,R),R) \, \, \text{exists} \right\} \in \Sigma.
	\end{equation*}
Theorem \ref{T: thermodynamic_limit} implies this event has full probability.    \end{remark} 

\begin{remark}  \label{R: silly remark} By definition of $\tilde{\varphi}^{\omega}(e,Q(0,R),R)$, the last limit in Theorem \ref{T: thermodynamic_limit}  implies that, for each $\rho > 0$,
\begin{equation*}
\lim_{R \to \infty} R^{1 - d} \tilde{\Phi}^{\omega}(e,0,Q^{e}(0,R \rho)) = \tilde{\varphi}(e) \rho^{d - 1} \quad \mathbb{P}\text{-almost surely}.  
\end{equation*}
\end{remark}

\begin{remark} \label{R: response} The same approach introduced in \cite{magnetic domains} and used in \cite{stochastic homogenization free discontinuity} also applies to the problem considered here.  However, it is necessary to choose a sufficiently nice boundary condition $q$.  

Assume that $q$ is smooth and, instead of \eqref{E: BC}, assume that there is an $L > 0$ such that 
\begin{equation*}
q(s) = -1 \quad \text{if} \, \, s \leq -L, \quad q(s) = 1 \quad \text{if} \, \, s \geq L.
\end{equation*}
In this case, $e(h) = 0$ if $h \geq L$.  As a consequence, one can obtain a sub-additive quantity by restricting attention to cubes in $L(\mathbb{Z}^{d -1} \oplus_{e} \mathbb{Z})$ and arguing as in \cite{stochastic homogenization free discontinuity}.  This leads to a slightly faster proof that the  limit $\lim_{N \to \infty} N^{1 - d} \tilde{\varphi}^{\omega}(e,Q(0,NL),NL)$ exists, without studying the commutativity of the limit $R \to \infty$, $h \to \infty$ as we have done here.

We emphasize that the proof given here does not depend on the choice of $q$ as long as it satisfies \eqref{E: BC}.       \end{remark}  

Now we turn to the proof of Lemma \ref{L: helpful_ergodic}:

\begin{proof}[Proof of Lemma \ref{L: helpful_ergodic}]  Define $\Omega_{\leq}$, $\Omega_{\ell}$, and $\Omega_{\text{horizon}}$ by 
\begin{align*}
\Omega_{\leq} &= \bigcap_{n,m \in \mathbb{N}} \left\{ \omega \in \Omega \, \mid \, 2^{(n + 1)(1- d)} \mathbb{E}(\tilde{\varphi}^{\omega}(e,Q(0,2^{n + 1}),m) \mid \Sigma_{e}) \leq \right. \\
			&\quad \quad \quad \left. 2^{n(1 - d)} \mathbb{E}(\tilde{\varphi}^{\omega}(e,Q(0,2^{n}),m) \mid \Sigma_{e}) \right\} \\
\Omega_{\ell} &= \left\{ \omega \in \Omega \, \mid \, 
\tilde{\varphi}^{\omega}_{\infty}(e) = \lim_{n \to \infty} 2^{n(1 -d)} \mathbb{E}(\tilde{\varphi}^{\omega}_{\infty}(e,Q(0,2^{n})) \mid \Sigma_{e}) \, \, \text{and} \, \, \right. \\
&\qquad \quad \left. \tilde{\varphi}^{\omega}(e,m) = \lim_{n \to \infty} 2^{n(1 - d)} \mathbb{E}(\tilde{\varphi}^{\omega}(e,Q(0,2^{n}),m)) \mid \Sigma_{e}) \, \, \text{if} \, \, m \in \mathbb{N} \right\} \\
\Omega_{\text{horizon}} &= \left\{ \omega \in \Omega \, \mid \, \mathbb{E}(\tilde{\varphi}_{\infty}^{\omega}(e,Q(0,2^{n})) \mid \Sigma_{e}) = \lim_{m \to \infty} \mathbb{E}(\tilde{\varphi}^{\omega}(e,Q(0,2^{n}),m) \mid \Sigma_{e}) \right\}.
\end{align*}
Since all the random variables involved are $\Sigma_{e}$-measurable, each of these events are also.  Therefore, it only remains to verify that $\mathbb{P}(\Omega_{\leq}) = \mathbb{P}(\Omega_{\ell}) = \mathbb{P}(\Omega_{\text{horizon}}) = 1$.  

We begin with $\Omega_{\leq}$.  Suppose $n,m \in \mathbb{N}$.  First, we use property (i) of Proposition \ref{P: important_properties} to find
\begin{equation*}
\tilde{\varphi}^{\omega}(e,Q(0,2^{n + 1}),m) \leq \sum_{y \in 2^{n -1} \{-1,1\}^{d - 1}} \tilde{\varphi}^{\omega}(e,Q(y,2^{n}),m) \quad \text{if} \, \, \omega \in \Omega.
\end{equation*}
By Lemma \ref{L: shifting and conditioning} in Appendix \ref{A: ergodic theory}, the following inequality holds almost surely:
\begin{equation*}
\mathbb{E}(\tilde{\varphi}^{\omega}(e,Q(0,2^{n + 1}),m) \mid \Sigma_{e}) \leq \sum_{y \in 2^{n - 1} \{-1,1\}^{d - 1}} \mathbb{E}(\tilde{\varphi}^{\tau_{- O_{e}(y)}\omega}(e,Q(y,2^{n},m) \mid \Sigma_{e}).
\end{equation*}
Applying property (ii) in Proposition \ref{P: important_properties}, we find
\begin{align*}
\mathbb{E}(\tilde{\varphi}^{\omega}(e,Q(0,2^{n + 1}),m) \mid \Sigma_{e}) &\leq \sum_{y \in 2^{n - 1} \{-1,1\}^{d - 1}} \mathbb{E}(\tilde{\varphi}^{\omega}(e,Q(0,2^{n}),m) \mid \Sigma_{e}) \\
	&= 2^{d - 1} \mathbb{E}(\tilde{\varphi}(e,Q(0,2^{n},m)) \mid \Sigma_{e}).
\end{align*}
Dividing by $2^{(n + 1)(d -1)}$ gives the inequality in the definition of $\Omega_{\leq}$.  We conclude that $\Omega_{\leq}$ contains the intersection of countably many events of full probability.  Therefore, $\mathbb{P}(\Omega_{\leq}) = 1$.  

To see that $\mathbb{P}(\Omega_{\ell}) = 1$, recall that if $n,m \in \mathbb{N}$, then the following equations hold almost surely by Propositions \ref{P: finite horizon limit} and \ref{P: ergodic_infinity}:
\begin{align*}
\tilde{\varphi}^{\omega}(e,m) &= \lim_{n \to \infty} 2^{n(1 - d)} \tilde{\varphi}^{\omega}(e,Q(0,2^{n},m)) \\
\tilde{\varphi}^{\omega}_{\infty}(e) &= \lim_{n \to \infty} 2^{n(1 - d)} \tilde{\varphi}^{\omega}_{\infty}(e,Q(0,2^{n})).
\end{align*}
Note that all of these quantities are bounded by $C_{\Lambda}$, and the random variables $\{\tilde{\varphi}^{\omega}(e,m)\}_{m \in \mathbb{N}}$ and $\tilde{\varphi}^{\omega}(e)$ are $\Sigma_{e}$-measurable.  Therefore, the conditional version of the dominated convergence theorem implies $\mathbb{P}(\Omega_{\ell}) = 1$.  

Since $\tilde{\varphi}^{\omega}_{\infty}(e,Q(0,2^{n})) = \lim_{m \to \infty} \tilde{\varphi}^{\omega}(e,Q(0,2^{n}),m)$ almost surely, the same arguments used to analyze $\Omega_{\ell}$ show that $\mathbb{P}(\Omega_{\text{horizon}}) = 1$.  
\end{proof}

\subsection{Other cubes} \label{S: other_cubes}  Before proceeding to the proof of $\Gamma$-convergence, we address the question of the thermodynamic limit of the finite-volume surface tension $\tilde{\Phi}^{\omega}$ defined in Section \ref{S: thermodynamic limit}.  We have already seen that the thermodynamic limit exists whenever we are dealing with cubes centered at the origin.  We now show that this still holds if we look at the blow-ups of any cube in $\mathbb{R}^{d}$.  

Since we are dealing with an uncountable family of cubes, we need to be wary of measurability issues.  In what follows, we start by analyzing cubes determined by points in $\mathbb{Q}^{d}$ and then recover the irrational ones using an approximation argument.

\begin{prop} \label{P: other_cubes 1}  There is an event $\hat{\Omega} \in \Sigma$ satisfying $\mathbb{P}(\hat{\Omega}) = 1$ such that if $\omega \in \hat{\Omega}$, $e \in S^{d-1} \cap \mathbb{R} \mathbb{Z}^{d}$, $x_{0} \in \mathbb{R}^{d}$, and $\rho > 0$, then 
\begin{equation*}
\tilde{\varphi}(e) \rho^{d - 1}= \lim_{R \to \infty} R^{1 - d} \tilde{\Phi}^{\omega}(e,Rx_{0},RQ^{e}(x_{0},\rho)).
\end{equation*} 
\end{prop}  

\begin{proof}  We begin with the case when $x_{0} \in \mathbb{Q}^{d}$ and $\rho \in \mathbb{Q}$.  The remaining cubes are treated using an approximation argument we sketch at the end of the proof.  

First, recall from Remark \ref{R: silly remark} that $\lim_{R \to \infty} R^{1 - d} \tilde{\Phi}^{\omega}(e, 0, Q^{e}(0,R\rho)) = \tilde{\varphi}(e) \rho^{d - 1}$ almost surely.  Therefore, Egoroff's Theorem implies there is an event $\Omega_{x_{0},\rho} \in \mathscr{B}$ such that $\mathbb{P}(\Omega_{x_{0},\rho}) > \frac{1}{2}$ and 
\begin{equation*}
\lim_{N \to \infty} \sup \left\{ |N^{1 - d} \tilde{\Phi}^{\omega}(e, 0, Q^{e}(0,N\rho)) - \tilde{\varphi}(e) \rho^{d - 1}| \, \mid \,  \omega \in \Omega_{x_{0},\rho} \right\} = 0.
\end{equation*}

Suppose $\delta \in \left(0,1\right) \cap \mathbb{Q}$.  By the ergodic theorem (see Theorem \ref{T: additive_ergodic_theorem} in Appendix \ref{A: ergodic theory}), there is an event $\hat{\Omega}_{\delta} \in \Sigma$ such that $\mathbb{P}(\hat{\Omega}_{\delta}) = 1$ and
\begin{equation*}
\lim_{R \to \infty} \frac{\mathcal{L}^{d}(\{y \in \mathbb{R}^{d} \, \mid \, \tau_{y} \omega \in \Omega_{x_{0},\rho}\} \cap R Q^{e}(x_{0},\delta \rho))}{R^{d}} > \frac{\delta^{d}\rho^{d}}{2} \quad \text{if} \, \, \omega \in \hat{\Omega}_{\delta}.
\end{equation*}
Let $\hat{\Omega}_{x_{0},\rho} = \bigcap_{\delta \in (0,1) \cap \mathbb{Q}} \hat{\Omega}_{\delta}$.  

We now show that 
\begin{equation} \label{E: tedium}  
\limsup_{R \to \infty} R^{1 - d} \tilde{\Phi}^{\omega}(e,Rx_{0},RQ^{e}(x_{0},\rho)) \leq \tilde{\varphi}(e) \rho^{d - 1} \quad \text{if} \, \, \omega \in \hat{\Omega}.
\end{equation}
We will use the fundamental estimate.  Henceforth assume $\omega \in \hat{\Omega}_{x_{0},\rho}$ is fixed and let $\delta \in \left(0,\frac{1}{2}\right) \cap \mathbb{Q}$ and $\alpha \in (0,\frac{1}{2})$ be free parameters.

In anticipation of our application of the estimate, we define the sets we will use.  These will depend on $\alpha$ and $\delta$, but we will omit this from the notation.  Let $U' = Q^{e}(x_{0},\rho)$.  By construction of $\hat{\Omega}_{x_{0},\rho}$, there is an $R_{0} > 0$ depending on $\omega$ such that if $R > R_{0}$, we can fix $y_{R} \in \mathbb{R}^{d}$ satisfying
\begin{equation*}
\tau_{y_{R}} \omega \in \Omega_{x_{0},\rho}, \quad y_{R} \in RQ^{e}(x_{0},\delta \rho)
\end{equation*} 
Observe that $\frac{y_{R}}{R} + (1-\delta) Q^{e}(0,\rho) \subseteq Q^{e}(x_{0},\rho)$.  Define families of sets $(\tilde{U}_{R},\tilde{V}_{R})_{R > R_{0}}$ by
\begin{equation*}
\tilde{U}_{R} = R^{-1}\left(y_{R} + \lfloor (1 - \delta - \alpha) R \rfloor Q^{e}(0,\rho)\right), \, \, \tilde{V}_{R} = U' \setminus \tilde{U}_{R}
\end{equation*}
As a matter of notational convenience, we define $\epsilon = R^{-1}$ and let $(U_{\epsilon},V_{\epsilon})_{\epsilon \in (0,R_{0}^{-1})}$ be given by $U_{\epsilon} = \tilde{U}_{R}$ and $V_{\epsilon} = \tilde{V}_{R}$.  Finally, define $E_{\delta} \subseteq \mathbb{R}^{d}$ by
\begin{equation*}
E_{\delta} = U' \cap \left\{x \in \mathbb{R}^{d} \, \mid \, |\langle x - x_{0}, e \rangle| < \frac{\delta \rho}{2} \right\}.
\end{equation*}
Notice that $\mathcal{L}^{d}(E_{\delta}) = \delta \rho^{d}$. 

Now we define the functions we will use.  Fix a minimizer $u_{R}$ of $\mathcal{F}^{\tau_{y_{R}}\omega}(\cdot; \lfloor (1 - \delta - \alpha) R \rfloor Q^{e}(0,\rho))$ with $u_{R} = q_{e}$ on the boundary.   Extend $u_{R}$ by letting it equal $q_{e}$ outside of $\lfloor (1-  \delta - \alpha)R \rfloor Q^{e}(0,\rho)$.  Finally, define $u_{\epsilon} \in H^{1}_{\text{loc}}(\mathbb{R}^{d};[-1,1])$ by 
\begin{equation*}
u_{\epsilon}(x) = u_{R} \left( \frac{x}{\epsilon} - y_{R} \right)
\end{equation*}
Observe that $\lim_{\epsilon \to 0^{+}} \|u_{\epsilon} - T_{x_{0}}q^{\epsilon}_{e}\|_{L^{1}(V_{\epsilon} \setminus E_{\delta})} = 0$.  Moreover, by construction,
\begin{align*}
\mathcal{F}^{\omega}_{\epsilon}(u_{\epsilon}; U_{\epsilon}) &= R^{1 - d} \mathcal{F}^{\tau_{y_{R}}\omega}(u_{R};  \lfloor (1 - \delta - \alpha) R \rfloor Q^{e}(0,\rho)) \\
	&= R^{1 - d} \tilde{\Phi}^{\tau_{y_{R}}\omega}(e,0, \lfloor (1 - \delta - \alpha) R \rfloor Q^{e}(0,\rho)).
\end{align*}

By the fundamental estimate, there is a family of cut-off functions $(\psi_{\epsilon})_{\epsilon \in (0,R_{0}^{-1})} \subseteq C_{c}^{\infty}(U'; [0,1])$ satisfying $\psi_{\epsilon} \equiv 1$ in $U_{\epsilon}$ and a constant $C_{\alpha} >0$ depending only on $W$, $\lambda$, $\Lambda$, and $\alpha$ such that as $\epsilon \to 0^{+}$, the following asymptotic holds:
\begin{equation*}
\mathcal{F}^{\omega}_{\epsilon}(\psi_{\epsilon} u_{\epsilon} + (1 - \psi_{\epsilon})T_{x_{0}}q^{\epsilon}_{e}; U_{\epsilon} \cup V_{\epsilon}) \leq \mathcal{F}^{\omega}_{\epsilon}(u_{\epsilon}; U') + \mathcal{F}^{\omega}_{\epsilon}(T_{x_{0}}q^{\epsilon}_{e}; V_{\epsilon}) + C_{\alpha} \mathcal{L}^{d}(E_{\delta}) + o(1)
\end{equation*}
Since $U_{\epsilon} \cup V_{\epsilon} = U' = Q^{e}(x_{0},\rho)$, we find
\begin{equation} \label{E: horrible}
\limsup_{R \to \infty} R^{1 - d} \tilde{\Phi}^{\omega}(e,Rx_{0}, RQ^{e}(x_{0},\rho)) \leq \limsup_{\epsilon \to 0^{+}} \left( \mathcal{F}^{\omega}_{\epsilon}(u_{\epsilon}; U') + \mathcal{F}^{\omega}_{\epsilon}(T_{x_{0}}q^{\epsilon}_{e}; V_{\epsilon}) \right) + C_{\alpha} \mathcal{L}^{d}(E_{\delta}) 
\end{equation}
Taking into account the geometry of $U' \setminus U_{\epsilon} = V_{\epsilon}$, we see that 
\begin{align*}
\limsup_{\epsilon \to 0^{+}} \mathcal{F}^{\omega}_{\epsilon}(u_{\epsilon}; U') &\leq  \limsup_{R \to \infty} R^{1 - d} \tilde{\Phi}^{\tau_{y_{R}}\omega}(e,0,\lfloor (1 - \delta) R \rfloor \hat{T}) + \gamma (\alpha,\delta) \\
\limsup_{\epsilon \to 0^{+}} \mathcal{F}^{\omega}_{\epsilon}(T_{x_{0}}q_{e}^{\epsilon} ; V_{\epsilon}) &\leq \gamma(\alpha,\delta),
\end{align*}
where $\gamma(\alpha,\delta)$ is given by
\begin{equation*}
\gamma(\alpha,\delta) = 2 (d - 1) C_{\Lambda} (2 \delta + \alpha) \rho^{d - 1}.
\end{equation*}  
Moreover, since $\tau_{y_{R}}\omega \in \Omega_{x_{0},\rho}$ for all $R > R_{0}$, the following identity holds:
\begin{align*}
\limsup_{R \to \infty} R^{1 - d} \tilde{\Phi}^{\tau_{y_{R}}\omega}(e,0,\lfloor (1 - \delta - \alpha) R \rfloor Q^{e}(0,\rho)) &= (1 - \delta - \alpha)^{d - 1} \tilde{\varphi}(e) \rho^{d - 1}.
\end{align*}
Thus, \eqref{E: tedium} follows after sending first $\delta \to 0^{+}$ and then $\alpha \to 0^{+}$ in \eqref{E: horrible}.  

The proof of the lower bound is obtained similarly: instead of approximating $T$ from inside by smaller cubes, we approximate it from the outside.  Thus, the statement of the proposition holds with $\hat{\Omega} = \bigcap_{x_{0} \in \mathbb{Q}^{d}} \bigcap_{\rho \in \mathbb{Q}_{+}} \hat{\Omega}_{x_{0},\rho}$, at least for cubes with center in $\mathbb{Q}^{d}$ and side length in $\mathbb{Q}_{+}$.  

If $x_{0} \in \mathbb{Q}^{d}$, but $\rho \notin \mathbb{Q}$, then $Q^{e}(x_{0},\rho)$ is well-approximated by cubes of the previous type.  In the previous arguments, if we replace $Q^{e}(0,\rho)$ by $Q^{e}(0,\zeta)$ for some $\zeta \in \mathbb{Q} \cap (0,\rho)$, then the same arguments show the the following asymptotic holds as $\zeta \to \rho$:
\begin{equation*}
\limsup_{R \to \infty} R^{1 - d} \tilde{\Phi}^{\omega}(e,Rx_{0},RQ^{e}(0,\rho)) \leq \tilde{\varphi}(e) \zeta^{d- 1} + o(1)
\end{equation*} 
Since $\zeta \in \mathbb{Q}$, we can do all of this without leaving the event $\hat{\Omega}$.  The same approach can be used to prove the lower bound.  

When $x_{0} \notin \mathbb{Q}^{d}$, we approximate $Q^{e}(x_{0},\rho)$ by $Q^{e}(x',\zeta)$, where $x' \in \mathbb{Q}^{d}$ and $\zeta \in \mathbb{Q}$.  As long as $\omega \in \hat{\Omega}$, we can use the fundamental estimate to show that if $Q^{e}(x',\zeta)$ is a large enough subset of $Q^{e}(x_{0},\rho)$, then minimizers in $Q^{e}(x',\zeta)$ can be interpolated to approximate minimizers in $Q^{e}(x_{0},\rho)$.  This gives the upper bound.  The lower bound is obtained by taking $Q^{e}(x',\zeta)$ to be a small enough superset.   
\end{proof}

Next, we will consider the case when $e \in S^{d - 1} \setminus \mathbb{R} \mathbb{Z}^{d}$.  Before doing so, it is helpful to know that the choice of $\{O_{e}\}_{e \in S^{d -1}}$ (see Section \ref{S: assumptions/notations}) was superfluous:

\begin{lemma} \label{L: orientation}  Let $\hat{\Omega}$ be the event obtained in Proposition \ref{P: other_cubes 1}.  If $O_{e}^{*}: \mathbb{R}^{d - 1} \to \langle e \rangle^{\perp}$ is a different choice of isometry, and if we denote by $\{Q^{e}_{*}(x,\rho) \, \mid \, x \in \mathbb{R}^{d}, \, \, \rho > 0\}$ the cubes obtained by replacing $O_{e}$ by $O_{e}^{*}$, then, for each $\omega \in \hat{\Omega}$, $x_{0} \in \mathbb{R}^{d}$, and $\rho > 0$,
	\begin{equation*}
		\tilde{\varphi}(e) \rho^{d - 1} = \lim_{R \to \infty} R^{1 - d} \Phi^{\omega}(e, Rx_{0}, RQ^{e}_{*}(x_{0},\rho)).
	\end{equation*}
\end{lemma}  

	\begin{proof}  Let $Q' = (O_{e}^{*})^{-1}(Q^{e}_{*}((x_{0},\rho))$ and set $t = \langle x_{0},e \rangle$.  The idea is to approximate $Q'$ using cubes oriented according to $(O_{e}^{*})^{-1} \circ O_{e}$ as in Riemannian integration.  It will then be easy to use basic estimates and Proposition \ref{P: other_cubes 1} to conclude.
	
	Recall that we have fixed the orthonormal basis $\{e_{1},\dots,e_{d - 1}\}$ of $\mathbb{R}^{d - 1}$ in Section \ref{S: cubes}.  Note that, by the definition of $O_{e}^{*}$ and $Q^{e}_{*}(x_{0},\rho)$, $Q'$ is a cube oriented according to this basis.  Define another orthonormal basis $\{v_{1},\dots,v_{d - 1}\}$ of $\mathbb{R}^{d-1}$ according to $v_{i} = (O_{e}^{*})^{-1}(O_{e}(e_{i}))$.  Notice that if $y \in \mathbb{R}^{d}$ and $\rho' > 0$, then $(O_{e}^{*})^{-1}(Q^{e}(y,\rho'))$ is a cube in $\mathbb{R}^{d - 1}$ oriented according $\{v_{1},\dots,v_{d - 1}\}$.
	
	Before proceeding further, it will be helpful to define the set operation $\oplus_{e}^{*}$ analogously to $\oplus_{e}$ (see Section \ref{S: euclidean}) but with $O_{e}^{*}$ replacing $O_{e}$.   
	
\textbf{Step 1:} upper bound
	
	Let $Q$ denote the unit cube in $\mathbb{R}^{d-1}$ oriented according to $\{v_{1},\dots,v_{d - 1}\}$.  For each $\delta > 0$, define the grid $\mathcal{J}^{-}_{\delta}$ by
		\begin{equation*}
			\mathcal{J}^{-}_{\delta} = \{w + \delta Q \, \mid \, \delta^{-1} w \in \text{span}_{\mathbb{Z}}\{v_{1},\dots,v_{d - 1}\}, \, \, w + \delta Q \subseteq Q'\}.
		\end{equation*}
and its trace $\mathscr{T}^{-}_{\delta} = \bigcup_{A \in \mathcal{J}^{-}_{\delta}} A$.  
Recall that since $Q'$ is a cube, it is Jordan measurable.  In particular, given $\zeta > 0$, we can fix a $\delta > 0$ such that $\mathcal{L}^{d - 1}(Q' \setminus \mathscr{T}^{-}_{\delta}) < \zeta$.    

	Notice that if $R, h > 0$ and $u \in H^{1}( R \mathscr{T}^{-}_{\delta} \oplus_{e}^{*} (R(t - \rho/2),R(t + \rho/2)) ; [-1,1])$ and $u = T_{x_{0}}q_{e}$ on the boundary, then the function $\tilde{u} : RQ^{*}_{e}(x_{0},\rho) \to [-1,1]$ defined by 
		\begin{equation*}
			\tilde{u}(x) = \left\{ \begin{array}{r l}
							u(x), & x \in R \mathscr{T}^{-}_{\delta} \oplus_{e}^{*} (-R\rho/2,R\rho/2) \\
							T_{x_{0}}q_{e}(x), & \text{otherwise}
						\end{array} \right.
		\end{equation*}
is in $H^{1}(RQ^{*}_{e}(x_{0},\rho))$, it equals $T_{x_{0}}q_{e}$ on the boundary, and, arguing as in Proposition \ref{P: basic_upper_bound}, we find
	\begin{align*}
		\mathcal{F}^{\omega}(\tilde{u}; RQ^{*}_{e}(x_{0},\rho)) \leq \mathcal{F}^{\omega}(u; R \mathscr{T}^{-}_{\delta} \oplus_{e}^{*} (R(t-\rho/2),R(r +\rho/2)) + C_{\Lambda} \zeta.  
	\end{align*}  
From this, we deduce that
	\begin{align*}
		\Phi^{\omega}(e,Rx_{0},Q^{e}_{*}(x_{0},\rho)) &\leq \Phi^{\omega}(e, Rx_{0},R \mathscr{T}^{-}_{\delta} \oplus_{e}^{*} (R(t -\rho/2),R(t +\rho/2)) + C_{\Lambda}\zeta \\
			&\leq \sum_{A \subseteq \mathcal{J}^{-}_{\delta}} \Phi^{\omega}(e, Rx_{0}, RA \oplus_{e}^{*} (R(t-\rho/2),R(t+\rho/2)) + C_{\Lambda} \zeta.
	\end{align*}
Notice that, by construction, there are points $\{y_{A} \, \mid \, A \in \mathcal{J}^{-}_{\delta}\}$ with $\langle y_{A}, e \rangle = t$ such that $A \oplus_{e}^{*} (t -\delta/2,t +\delta/2) = Q^{e}(y_{A}, \delta)$.  Moreover, since $\delta < \rho$, the inequality \eqref{E: almost_decreasing} implies
	\begin{align*}
		&\limsup_{R \to \infty} R^{1-d} \Phi^{\omega}(e,Rx_{0}, RA \oplus_{e}^{*} (R(t-\rho/2),R(t+\rho/2)) \\
		&\qquad \leq \limsup_{R \to \infty} R^{1-d} \Phi^{\omega}(e,Rx_{0}, R Q^{e}(y_{A})).
	\end{align*}
Substituting this into the previous estimates, dividing by $R^{d - 1}$, and sending $R \to \infty$, we find
	\begin{align*}
		\limsup_{R \to \infty} R^{1- d} \Phi^{\omega}(e,Rx_{0},Q^{e}_{*}(x_{0},\rho)) &\leq \sum_{A \subseteq \mathcal{J}^{-}_{\delta}} \limsup_{R \to \infty} R^{1 - d} \Phi^{\omega}(e,Ry_{A},RQ^{e}(y_{A},\delta)) + C_{\Lambda} \zeta \\
			&= \tilde{\varphi}(e) \mathcal{L}^{d - 1}(\mathscr{T}^{-}_{\delta}) + C_{\Lambda} \zeta.
	\end{align*}
Since $\zeta > 0$ was arbitrary and $\lim_{\delta \to 0^{+}} \mathcal{L}^{d-1}(\mathscr{T}^{-}_{\delta}) = \rho^{d-1}$, this yields
	\begin{equation*}
		\limsup_{R \to \infty} R^{1 - d} \Phi^{\omega}(e,Rx_{0},RQ^{e}_{*}(x_{0},\rho)) \leq \tilde{\varphi}(e) \rho^{d - 1}.
	\end{equation*}
	
\textbf{Step 2:} interlude on disjoint unions of cubes

	In what follows, the following fact will be helpful.  Suppose $\{w_{1},\dots,w_{N}\} \subseteq \delta \text{span}_{\mathbb{Z}}\{v_{1},\dots,v_{d - 1}\}$, and let $J = \bigcup_{i = 1}^{N} (w_{i} + \delta Q)$.  We claim that if $t \in \mathbb{R}$ and $\omega \in \hat{\Omega}$, then
		\begin{equation*}
			\lim_{R \to \infty} R^{1 - d} \Phi^{\omega}(e,Rte, RJ \oplus_{e}^{*} (R(t-\delta/2),R(t+\delta/2))) = \tilde{\varphi}(e) \mathcal{L}^{d - 1}(J).
		\end{equation*}
		
	To see this, first pick $x_{1},\dots,x_{N} \in \mathbb{R}^{d}$ such that $\langle x_{i}, e \rangle = t$ and
		\begin{equation*}
			Q^{e}(x_{i},\delta) = (w_{i} + \delta Q) \oplus_{e}^{*} (t-\delta/2,t+\delta/2)
		\end{equation*}  
	Next, observe that, by sub-additivity,
		\begin{align*}
			\Phi^{\omega}(e, Rte, RJ \oplus_{e}^{*} (R(t -\delta/2),R(t + \delta/2))) &\leq \sum_{i = 1}^{N} \Phi^{\omega}(e,Rx_{i},Q^{e}(x_{i},\delta)) 
		\end{align*}
	Thus, dividing by $R^{d-1}$, sending $R \to \infty$, and invoking Proposition \ref{P: other_cubes 1}, we find
		\begin{align*}
			\limsup_{R \to \infty} R^{1 - d} \Phi^{\omega}(e, Rte, RJ \oplus_{e}^{*} (R(t -\delta/2),R(t + \delta/2)))
				&\leq \tilde{\varphi}(e) N \delta^{d - 1}  \\
				&= \tilde{\varphi}(e) \mathcal{L}^{d- 1}(J).
		\end{align*}
		
	Next, let $\mathcal{J} = \{w + \delta Q \, \mid \, w \in \delta \text{span}_{\mathbb{Z}}\{v_{1},\dots,v_{d-1}\} \cap [-M,M)^{d}\}$ and let $\overline{Q}$ be the cube in $\mathbb{R}^{d-1}$ given by 
		\begin{equation*}
			\overline{Q} = \bigcup_{w \in \mathcal{J}} (w + \delta Q)
		\end{equation*}
	where $M \in \mathbb{N}$ is chosen so that $J \subseteq \overline{Q}$.  Let $\mathcal{J}_{1} = \{Q_{0} \in \mathcal{J} \, \mid \, Q_{0} \subseteq J\}$ denote the sub-set of $\mathcal{J}$ consisting of cubes contained in $J$.  
	
	By sub-additivity and Proposition \ref{P: other_cubes 1}, we find
		\begin{align*}
			\tilde{\varphi}(e) \mathcal{L}^{d - 1}(\overline{Q}) &= \lim_{R \to \infty} R^{1 - d} \Phi^{\omega}(e, Rte, R\overline{Q} \oplus_{e}^{*} (R(t - M),R(t + M))) \\
			&\leq \liminf_{R \to \infty} R^{1-d} \Phi^{\omega}(e,Rte, RJ \oplus_{e}^{*} (R(t - M), R(t + M))) \\
			&\quad + \sum_{Q \in \mathcal{J} \setminus \mathcal{J}_{1}} \limsup_{R \to \infty} R^{1 -d} \Phi^{\omega}(e,Rte, RQ \oplus_{e}^{*} (R(t -M),R(t + M))) \\
			&\leq \liminf_{R \to \infty} R^{1-d} \Phi^{\omega}(e,Rte,RJ \oplus_{e}^{*} (R(t -\delta/2),R(t + \delta/2)) \\
			&\quad + \sum_{Q \in \mathcal{J} \setminus \mathcal{J}_{1}} \lim_{R \to \infty} R^{1 - d} \Phi^{\omega}(e,Rte,RQ \oplus_{e}^{*}(R(t - \delta/2),R(t + \delta/2))) \\
			&=  \liminf_{R \to \infty} \Phi^{\omega}(e,Rte,RJ \oplus_{e}^{*} (R(t -\delta/2),R(t + \delta/2))  + \tilde{\varphi}(e) \sum_{Q \in \mathcal{J} \setminus \mathcal{J}_{1}} \delta^{d - 1}.
		\end{align*}
Thus, since $\sum_{Q \in \mathcal{J} \setminus \mathcal{J}_{1}} \delta^{d -1} = \mathcal{L}^{d -1}(\overline{Q} \setminus J)$ and $J \subseteq \overline{Q}$,
	\begin{equation*}
		\tilde{\varphi}(e) \mathcal{L}^{d -1}(J) \leq  \liminf_{R \to \infty} \Phi^{\omega}(e,Rte,RJ \oplus_{e}^{*} (R(t -\delta/2),R(t + \delta/2))).
	\end{equation*}
	
\textbf{Step 3:} lower bound

Finally, define $\mathcal{J}^{+}_{\delta}$ by
	\begin{equation*}
		\mathcal{J}^{+}_{\delta} = \{w + \delta Q \, \mid \, \delta^{-1} w \in \text{span}_{\mathbb{Z}}\{v_{1},\dots,v_{d - 1}\}, \, \, (w + \delta Q) \cap Q' \neq \phi\}
	\end{equation*}
and define its trace $\mathscr{T}^{+}_{\delta}$ by $\mathscr{T}^{+}_{\delta} = \bigcup_{A \in \mathcal{J}^{+}_{\delta}} A$.  As in Step 1, given $\zeta > 0$, we can fix $\delta > 0$ such that $\mathcal{L}^{d-1}(\mathscr{T}^{+}_{\delta} \setminus Q') < \zeta$.  To conclude, we will use Step 2 to estimate the energy in $R\mathscr{T}^{+}_{\delta}$ as $R \to \infty$ and then use an appropriate choice of configuration in $Q'$ to conclude.

For each $R > 0$, we can fix $u_{R} \in H^{1}(RQ_{e}^{*}(x_{0},\rho))$ such that $u = T_{x_{0}}q_{e}$ on the boundary and $\mathcal{F}^{\omega}(u_{R} ; RQ_{e}^{*}(x_{0},\rho)) = \Phi^{\omega}(e, Rx_{0}, RQ_{e}^{*}(x_{0},\rho))$.  Extend $u_{R}$ to $R\mathscr{T}^{+}_{\delta} \oplus_{e}^{*} (R(t - \rho),R(t + \rho))$ by letting $u_{R}(x) = T_{x_{0}}q_{e}(x)$ if $x \notin RQ_{e}^{*}(x_{0},\rho)$.  With this definition, $u_{R} \in H^{1}(R\mathscr{T}^{+}_{\delta} \oplus_{e} (R(t - \rho),R(t + \rho)))$, $u_{R} = T_{x_{0}}q_{e}$ on the boundary, and, thus,
	\begin{align*}
		\Phi^{\omega}(e, Rx_{0}, RQ_{e}^{*}(x_{0},\rho)) &= \mathcal{F}^{\omega}(u_{R} ; RQ_{e}^{*}(x_{0},\rho)) \\
								&\geq \mathcal{F}^{\omega}(u_{R}; R\mathscr{T}^{+}_{\delta} \oplus_{e}^{*} (R(t - \rho),R(t + \rho))) - C_{\Lambda} \mathcal{L}^{d-1}(\mathscr{T}^{+}_{\delta} \setminus Q') \\
								&\geq \Phi^{\omega}(e,Rte,R\mathscr{T}^{+}_{\delta} \oplus_{e}^{*} (R(t - \rho),R(t + \rho))) - C_{\Lambda} \zeta
	\end{align*}
Renormalizing, sending $R \to \infty$, and appealing to the result of Step 2, we find
	\begin{align*}
		\liminf_{R \to \infty} R^{1- d} \Phi^{\omega}(e, Rx_{0}, RQ_{e}^{*}(x_{0},\rho)) &\geq \tilde{\varphi}(e) \mathcal{L}^{d-1}(\mathscr{T}^{+}_{\delta}) - C_{\Lambda} \zeta
	\end{align*} 
In the limit $\zeta \to 0^{+}$, this yields
	\begin{equation*}
		\liminf_{R \to \infty} R^{1- d} \Phi^{\omega}(e, Rx_{0}, RQ_{e}^{*}(x_{0},\rho)) \geq \tilde{\varphi}(e) \mathcal{L}^{d-1}(Q')
	\end{equation*}

\end{proof}  

\begin{prop} \label{P: other_cubes 2}  Let $\hat{\Omega}$ be the event defined in Proposition \ref{P: other_cubes 1}.  If $\omega \in \hat{\Omega}$, $e \in S^{d -1}$, $x_{0} \in \mathbb{R}^{d}$, and $\rho > 0$, then
\begin{equation} \label{E: what_we_want}
\tilde{\varphi}(e) \rho^{d - 1}= \lim_{R \to \infty} R^{1 - d} \tilde{\Phi}^{\omega}(e,Rx_{0},RQ^{e}(0,\rho)).
\end{equation}
\end{prop}

\begin{proof}  To simplify the exposition, we will assume that $x_{0} = 0$ and $\rho = 1$.  The case $\rho \neq 1$ is no different, other than the additional factors of $\rho^{d - 1}$ that need to be tracked.  The same argument works with minor modifications when $x_{0} \neq 0$.

Fix $e \in S^{d - 1}$.  Choose $(e_{n})_{n \in \mathbb{N}} \subseteq S^{d-1} \cap \mathbb{R} \mathbb{Z}^{d}$ such that $\lim_{n \to \infty} e_{n} = e$.  Precomposing with an orthogonal transformation of $\mathbb{R}^{d-1}$ if necessary, we can assume that $O_{e_{n}} \to O_{e}$ in the operator norm as $n \to \infty$.  (Though this changes the orientation of the cubes we will use, Lemma \ref{L: orientation} shows this does not affect the validity of Proposition \ref{P: other_cubes 1}.)  Let $\alpha \in \left(0,1\right)$ and observe that if $n$ is sufficiently large, then 
\begin{equation} \label{E: geometric observation}
Q^{e_{n}}\left(0,\alpha \right) \subseteq Q^{e}(0,1)
\end{equation} 
In order to apply the fundamental estimate, we define a family of $\epsilon$-independent open sets as follows:
\begin{equation*}
U_{n,\alpha} = Q^{e_{n}} \left(0,\alpha \right), \quad V_{n,\alpha} = Q^{e}(0,1) \setminus U_{n,\alpha}, \quad U' = Q^{e}(0,1).
\end{equation*}

We now define the functions to be used in the estimate.  For each $n$, $\alpha$, and $\epsilon > 0$, let $u_{n,\alpha}^{\epsilon} \in H^{1}(U_{n,\alpha};[-1,1])$ be a minimizer of $\mathcal{F}^{\omega}_{\epsilon}(\cdot ; U_{n,\alpha})$ subject to $u_{n,\alpha}^{\epsilon} = q_{e_{n}}^{\epsilon}$ on $\partial U_{n,\alpha}$.  Extend $u_{n,\alpha}^{\epsilon}$ to $\mathbb{R}^{d}$ by setting $u_{n,\alpha}^{\epsilon} = q_{e_{n}}^{\epsilon}$ in $\mathbb{R}^{d} \setminus U_{n,\alpha}$.  

Observe that the family $(u_{n,\alpha}^{\epsilon})_{\epsilon > 0}$ satisfies $\lim_{\epsilon \to 0^{+}} \|u_{n,\alpha}^{\epsilon} - \chi_{e_{n}}\|_{L^{1}(V_{n,\alpha})} = 0$.  In particular, $\lim_{\epsilon \to 0^{+}} \|u_{n,\alpha}^{\epsilon} - q^{\epsilon}_{e}\|_{L^{1}(V_{n,\alpha} \setminus E_{n})} = 0$, where $E_{n}$ is given by
\begin{equation*}
E_{n} = \left\{ x \in \mathbb{R}^{d} \, \mid \, \langle x, e_{n} \rangle < 0 < \langle x, e \rangle\right\} \cup \left\{ x \in \mathbb{R}^{d} \, \mid \, \langle x, e \rangle < 0 < \langle x, e_{n} \rangle \right\}.
\end{equation*}
Let $\zeta_{n} = \mathcal{L}^{d}(V_{n,\alpha} \cap E_{n})$ and observe that $\lim_{n \to \infty} \zeta_{n} = 0$.  

We now apply the fundamental estimate: we fix a family $(\psi^{\epsilon}_{n,\alpha})_{\epsilon > 0} \subseteq C^{\infty}_{c}(U'; [0,1])$ satisfying $\psi^{\epsilon}_{n,\alpha} \equiv 1$ in $U_{n,\alpha}$ and a constant $C_{\alpha} > 0$ such that as $\epsilon \to 0^{+}$, the following asymptotic holds:
\begin{equation*}
\mathcal{F}^{\omega}_{\epsilon}(\psi^{\epsilon}_{n,\alpha} u_{n,\alpha}^{\epsilon} + (1 - \psi^{\epsilon}_{n,\alpha}) q^{\epsilon}_{e} ; V_{n,\alpha} \cup U_{n,\alpha}) \leq \mathcal{F}^{\omega}_{\epsilon}(u_{n,\alpha}^{\epsilon}; U_{n,\alpha}') + \mathcal{F}^{\omega}_{\epsilon}(q^{\epsilon}_{e}; V') + C_{\alpha} \zeta_{n} + o(1),
\end{equation*}
where $C_{\alpha}$ depends only on $\alpha$ and $W$.  Sending $\epsilon \to 0^{+}$ and appealing to \eqref{E: geometric observation}, we obtain
\begin{align} \label{E: key_continuity}
\limsup_{\epsilon \to 0^{+}} \epsilon^{d - 1} \tilde{\varphi}^{\omega}(e, Q(0,\epsilon^{-1}), \epsilon^{-1}) &\leq \lim_{\epsilon \to 0^{+}} \epsilon^{d - 1} \tilde{\varphi}^{\omega} \left(e_{n},Q \left(0,\alpha \epsilon^{-1}\right), \alpha \epsilon^{-1} \right) \\
	&\quad + \limsup_{\epsilon \to 0^{+}} \mathcal{F}^{\omega}_{\epsilon}(q_{e}^{\epsilon}; V_{n,\alpha}) + C_{\alpha} \zeta_{n}. \nonumber
\end{align}   

Observe that 
\begin{equation*}
\lim_{\alpha \to 1^{-}} \lim_{n \to \infty} \limsup_{\epsilon \to 0^{+}} \mathcal{F}^{\omega}_{\epsilon}(q^{\epsilon}_{e}; V_{n,\alpha}) = 0
\end{equation*}  
and
\begin{equation*}
\lim_{\alpha \to 1^{-}} \lim_{n \to \infty} C_{\alpha} \zeta_{n} = \lim_{\alpha \to 1^{-}} (C_{\alpha} \cdot 0) = 0.
\end{equation*}
Moreover, since $\omega \in \hat{\Omega}$,
\begin{equation} 
\lim_{\epsilon \to 0^{+}} \epsilon^{d - 1} \tilde{\varphi}^{\omega} \left(e_{n},Q\left(0,\alpha \epsilon^{-1} \right), \alpha \epsilon^{-1} \right) = \alpha^{d - 1} \tilde{\varphi}(e_{n})
\end{equation}
no matter the choice of $\alpha$ or $n$.
Thus, writing $R = \epsilon^{-1}$ in \eqref{E: key_continuity} and sending $n \to \infty$ first and then $\alpha \to 1^{-}$, we find
\begin{equation} \label{E: continuity_down}
\limsup_{R \to \infty} R^{1 - d} \tilde{\varphi}^{\omega}(e,Q(0,R),R) \leq \liminf_{n \to \infty} \tilde{\varphi}(e_{n}).
\end{equation}

Repeating the previous argument with the roles of $e_{n}$ and $e$ reversed, we obtain
\begin{equation} \label{E: continuity_up}
\limsup_{n \to \infty} \tilde{\varphi}(e_{n}) \leq \liminf_{R \to \infty} R^{1 - d} \tilde{\varphi}^{\omega}(e,Q(0,R),R) = \tilde{\varphi}^{\omega}(e).
\end{equation}
Combining \eqref{E: continuity_up} and \eqref{E: continuity_down}, we conclude 
\begin{equation} \label{E: continuity}
\liminf_{n \to \infty} \tilde{\varphi}(e_{n}) = \tilde{\varphi}^{\omega}(e) = \limsup_{R \to \infty} R^{1 - d} \tilde{\varphi}^{\omega}(e,Q(0,R),R) = \limsup_{n \to \infty} \tilde{\varphi}(e_{n}).
\end{equation}
Observing that $\mathbb{P}(\tilde{\Omega}_{e} \cap \hat{\Omega}) = 1$, we see that $\lim_{n \to \infty} \tilde{\varphi}(e_{n})$ exists and equals $\tilde{\varphi}(e)$.  Therefore, \eqref{E: continuity} implies $\tilde{\varphi}^{\omega}(e) = \tilde{\varphi}(e)$.  
Since $\tilde{\varphi}^{\omega}(e,Q(0,R),R) = \tilde{\Phi}^{\omega}(e,0,Q^{e}(0,R))$, the middle equality in \eqref{E: continuity} is exactly what we sought to prove.  
\end{proof} 

Rerunning the arguments of the previous proof, we find that $\tilde{\varphi} : S^{d -1} \to (0,\infty)$ is continuous:

\begin{prop} \label{P: continuity} $\tilde{\varphi}^{\omega} : S^{d -1} \to (0,\infty)$ is continuous.  \end{prop}  

\begin{proof}  Assume $(e_{n})_{n \in \mathbb{N}} \subseteq S^{d - 1}$ and $\lim_{n \to \infty} e_{n} = e$.  We can repeat the arguments of the previous proof: by Proposition \ref{P: other_cubes 2}, the restriction that $(e_{n})_{n \in \mathbb{N}}$ be in $\mathbb{R} \mathbb{Z}^{d}$ is no longer necessary so the arguments go through.  In particular, \eqref{E: continuity} implies $\tilde{\varphi}(e) = \lim_{n \to \infty} \tilde{\varphi}(e_{n})$.  Since the approximating sequence and the point $e$ were arbitrary, $\tilde{\varphi}^{\omega}$ is continuous.  \end{proof}  

\begin{remark}  A direct argument in the spirit of Propositions \ref{P: other_cubes 2} and \ref{P: continuity} can be used to show that the one-homogeneous extension of $\tilde{\varphi}$ is convex.  Since this already follows from results we cite later, we will not present the proof.  However, the reader can find proofs along these lines in the work of Messenger, Miracle-Sole, and Ruiz on lattice systems \cite{miracle-sole} and the work of Caffarelli and de la Llave on interfaces in periodic media \cite{caffarelli de la llave}.  \end{remark}  

Finally, we observe that Theorem \ref{T: other_cubes} is proved:

\begin{proof}[Proof of Theorem \ref{T: other_cubes}]  Everything follows from Proposition \ref{P: other_cubes 2}.  \end{proof}    

\section{$\Gamma$-convergence} \label{S: geometry}

We now prove Theorem \ref{T: main result}.  As advertised in the introduction, we use the machinery already developed in \cite{periodic_paper}.  The key point is the existence of the limit in the definition of $\tilde{\varphi}$ implies $\Gamma$-convergence.


In our proof of Theorem \ref{T: main result}, it will be helpful to know that $\Gamma$-convergence holds when $u = T_{x} \chi_{e}$ for some $x \in \mathbb{R}^{d}$ and $e \in S^{d -1}$.  This is a straightforward consequence of what we have already proved and the fundamental estimate.  In fact, this follows easily from \cite[Theorem 3.7]{periodic_paper}.  For the convenience of the reader, we state and prove it here.  The proof will be deferred until after that of Theorem \ref{T: main result}.  

\begin{prop} \label{P: planar_convergence} Suppose $e \in S^{d - 1}$, $x_{0} \in \mathbb{R}^{d - 1}$, $s_{0} \in \mathbb{R}$, $r, t > 0$, and $\omega \in \hat{\Omega}$.  If $A = Q(x_{0},r) \oplus_{e} (s_{0} - t, s_{0} + t))$, then 
\begin{align} \label{E: planar convergence}
			\tilde{\varphi}(e) r^{d - 1} &= \lim_{\delta \to 0^{+}} \liminf_{\epsilon \to 0^{+}} \inf \left\{ \mathcal{F}^{\omega}_{\epsilon}(u; A) \, \mid \, \|u - T_{s_{0} e}\chi_{e}\|_{L^{1}(A)} < \delta \right\} \\
				&=\lim_{\delta \to 0^{+}} \limsup_{\epsilon \to 0^{+}} \inf \left\{ \mathcal{F}^{\omega}_{\epsilon}(u; A) \, \mid \, \|u - T_{s_{0} e}\chi_{e}\|_{L^{1}(A)} < \delta \right\} \nonumber
\end{align}
\end{prop}  

Now we recall the machinery developed in \cite{periodic_paper}.  In what follows, we denote by $c_{W}$ the normalization constant
	\begin{equation*}
		c_{W} = \int_{-1}^{1} \sqrt{W(u)} \, du.
	\end{equation*}

\begin{theorem}[Theorems 3.3 and 3.5, \cite{periodic_paper}]  \label{T: abstract_machine} If $\omega \in \Omega$ and $(\epsilon_{j})_{j \in \mathbb{N}} \subseteq (0,\infty)$ satisfies $\lim_{j \to \infty} \epsilon_{j} = 0$, then there is a subsequence $(\epsilon_{j_{k}})_{k \in \mathbb{N}}$ and a functional $\mathscr{I}$ such that $\mathcal{F}^{\omega}_{\epsilon_{j_{k}}} \overset{\Gamma} \to \mathscr{I}$.  In fact, there is a Borel function $\tilde{\psi} : \mathbb{R}^{n} \times S^{d - 1} \to [0,\infty)$ such that
\begin{equation*}
\sqrt{\lambda} c_{W} \leq \tilde{\psi}(x,e) \leq \sqrt{\Lambda} c_{W} \quad \text{if} \, \, x \in \mathbb{R}^{d}, \, \, e \in S^{d - 1} 
\end{equation*}
and
\begin{equation*}
\mathscr{I}(u; A) = \left\{ \begin{array}{ r l }
			\int_{\partial^{*} \{u = 1\}} \tilde{\psi}(\xi,\nu_{\{u = 1\}}(\xi)) \, \mathcal{H}^{d - 1}(d \xi), & u \in BV(A; \{-1,1\}) \\
			+ \infty, & \text{otherwise}.
		\end{array} \right.
\end{equation*}
Moreover, $\tilde{\psi}$ can be recovered from $\mathscr{I}$ via the following derivation formula: if $x \in \mathbb{R}^{d}$ and $e \in S^{d - 1}$, then
\begin{equation} \label{E: derivation}
\tilde{\psi}(x,e) = \limsup_{\rho \to 0^{+}} \rho^{1 - d} \inf \{ \mathscr{I}(u; Q^{e}(x,\rho)) \, \mid \, u = T_{x}\chi_{e} \, \, \text{in} \, \, \mathbb{R}^{d} \setminus Q^{e}(x,\rho) \}.
\end{equation}
\end{theorem}

\begin{remark} \label{R: neighborhood_trick}  In the derivation formula \eqref{E: derivation}, it suffices to consider functions $u \in L^{1}_{\text{loc}}(\mathbb{R}^{d})$ such that $u = T_{x}\chi_{e}$ in a neighborhood of $\mathbb{R}^{d} \setminus Q^{e}(x,\rho)$.  See \cite[Equation 4.5]{periodic_paper}.  \end{remark}  

We now apply Proposition \ref{P: planar_convergence} and Theorem \ref{T: abstract_machine} to prove Theorem \ref{T: main result}:

\begin{proof}[Proof of Theorem \ref{T: main result}]   Statement (i) is more-or-less immediate from the lower bound $\sqrt{\lambda} \|\cdot\| \leq \varphi^{\omega}(x,\cdot)$ in $\mathbb{R}^{d}$.  From this, we see that if $(u_{\epsilon})_{\epsilon > 0} \subseteq H^{1}(A; [-1,1])$, then
	\begin{equation*}
		\sup \left\{ \int_{A} \left( \frac{\lambda \epsilon\|Du(x)\|^{2}}{2} + \epsilon^{-1}W(u(x)) \right) \, dx \, \mid \, \epsilon > 0\right\} \leq \sup \left\{\mathcal{F}^{\omega}_{\epsilon}(u_{\epsilon};A) \, \mid \, \epsilon > 0\right\}.
	\end{equation*}
Therefore, (i) follows from known results in the spatially homogeneous context (cf.\ \cite{alberti guide}).  

Now we turn to (ii) and (iii).  To see that $\mathcal{F}^{\omega}_{\epsilon} \overset{\Gamma}\to \mathscr{E}$, it is enough to show that any sequence $(\epsilon_{j})_{j \in \mathbb{N}}$ satisfying $\lim_{j \to \infty} \epsilon_{j} = 0$ has a further subsequence $(\epsilon_{j_{k}})_{k \in \mathbb{N}}$ such that $\mathcal{F}^{\omega}_{\epsilon_{j_{k}}} \overset{\Gamma}\to \mathscr{E}$ as $k \to \infty$.  In particular, it only remains to show that if $\mathscr{I}$ is obtained as in Theorem \ref{T: abstract_machine}, then $\mathscr{I} = \mathscr{E}$.  Equivalently, we will show that the function $\tilde{\psi}$ in that theorem satisfies $\tilde{\psi}(x,\cdot) = \tilde{\varphi}$ in $\mathbb{R}^{d}$.    

To avoid clunky notation, we won't distinguish between the sequence $(\epsilon_{j})_{j \in \mathbb{N}}$ and its subsequence in the sequel, that is, we will write $(\epsilon_{j})_{j \in \mathbb{N}}$ instead of $(\epsilon_{j_{k}})_{k \in \mathbb{N}}$.

Fix $x \in \mathbb{R}^{d}$.  First, we show that $\tilde{\psi}(x,\cdot) \leq \tilde{\varphi}$ pointwise.  Suppose $e \in S^{d-1}$.  By setting $u = T_{x}\chi_{e}$ in \eqref{E: derivation}, we find
\begin{equation*}
\tilde{\psi}(x,e) \leq \limsup_{\rho \to 0^{+}} \rho^{1 - d} \mathscr{I}(T_{x}\chi_{e}; Q^{e}(x,\rho)).
\end{equation*}
By Proposition \ref{P: planar_convergence}, $\mathscr{I}(T_{x}\chi_{e}; Q^{e}(x,\rho)) = \tilde{\varphi}(e) \rho^{d - 1}$ since $\omega \in \hat{\Omega}$.  Thus, we obtain the upper bound
\begin{equation*}
\tilde{\psi}(x,e) \leq \tilde{\varphi}(e).
\end{equation*}

The complementary inequality follows from our work on the thermodynamic limit.  By Remark \ref{R: neighborhood_trick}, it suffices to show that if $u \in L^{1}_{\text{loc}}(\mathbb{R}^{d})$ and $u = T_{x}\chi_{e}$ in $\mathbb{R}^{d} \setminus \overline{Q^{e}(x, 
\alpha \rho)}$ for some $\alpha \in (0,1)$, then
\begin{equation} \label{E: abstract_lower_bound}
\mathscr{I}(u; Q^{e}(x,\rho)) \geq \tilde{\varphi}(e) \rho^{d -1}.
\end{equation}  

Fix such a $u$ and $\alpha$.  Since \eqref{E: abstract_lower_bound} holds trivially otherwise, we can assume 
	\begin{equation*}
		u \in BV(Q^{e}(x,\rho); \{-1,1\}), \quad \mathscr{I}(u;Q^{e}(x,\rho)) < \infty.
	\end{equation*}
Since $\mathcal{F}^{\omega}_{\epsilon_{j}} \overset{\Gamma}\to \mathscr{I}$, we can fix a sequence $(u_{\epsilon_{j}})_{j \in \mathbb{N}} \subseteq H^{1}(Q^{e}(x,\rho); [-1,1])$ satisfying the following conditions:
\begin{align*}
&\lim_{j \to \infty} \|u_{\epsilon_{j}} - u\|_{L^{1}(Q^{e}(x,\rho))} = 0 \\
&\mathscr{I}(u; Q^{e}(x,\rho)) = \lim_{j \to \infty} \mathcal{F}^{\omega}_{\epsilon_{j}}(u_{\epsilon_{j}}; Q^{e}(x,\rho))
\end{align*}
We claim that \eqref{E: abstract_lower_bound} now follows easily from the fundamental estimate.  

Indeed, for each $\beta \in (\alpha,1)$, introduce $\epsilon$-independent open sets $U'_{\beta} = Q^{e}(x,\rho)$, $U_{\beta} = Q^{e}(x, \beta \rho)$, and $V_{\beta} = U'_{\beta} \setminus \overline{U_{\beta}}$.  Since $u = T_{x}\chi_{e}$ in $V_{\beta}$, the fundamental estimate applies: there is a family of cut-off functions $(\psi_{j,\beta})_{j \in \mathbb{N}} \subseteq C^{\infty}_{c}(U'_{\beta}; [0,1])$ such that $\psi_{j,\beta} \equiv 1$ in $U_{\beta}$ and
\begin{equation*}
\mathcal{F}^{\omega}_{\epsilon_{j}}(\psi_{j,\beta} u_{\epsilon_{j}} + (1 - \psi_{j,\beta}) T_{x}q^{\epsilon_{j}}_{e}; U'_{\beta}) \leq \mathcal{F}^{\omega}_{\epsilon_{j}}(u_{\epsilon_{j}}; U'_{\beta}) + \mathcal{F}^{\omega}_{\epsilon_{j}}(T_{x}q^{\epsilon_{j}}_{e}; V_{\beta}) + o(1)
\end{equation*}   
as $j \to \infty$.  By the definition of $U'_{\beta}$ and our results on the thermodynamic limit, we find
\begin{equation*}
\tilde{\varphi}(e) \rho^{d - 1} \leq \mathscr{I}(u; Q^{e}(x,\rho)) + \limsup_{j \to \infty} \mathcal{F}^{\omega}_{\epsilon_{j}}(T_{x}q_{e}^{\epsilon_{j}} ;V_{\beta}).
\end{equation*}
Sending $\beta \to 1^{-}$ and observing that $\lim_{\beta \to 1^{-}} \limsup_{j \to \infty} \mathcal{F}^{\omega}_{\epsilon_{j}}(T_{x}q^{\epsilon_{j}}_{e}; V_{\beta}) = 0$, we conclude
\begin{equation*}
\tilde{\varphi}(e) \rho^{d - 1} \leq \mathscr{I}(u; Q^{e}(x,\rho)).
\end{equation*}
Since this is true independently of the choice of $u$ and $\alpha$, \eqref{E: derivation} implies the inequality
\begin{equation*}
\tilde{\varphi}(e) \leq \tilde{\psi}(x,e).
\end{equation*}

	The results of the previous two paragraphs together show that $\tilde{\varphi}(e) = \tilde{\psi}(x,e)$.  Since $x$ and $e$ were arbitrary, we conclude $\mathscr{I} = \mathscr{E}$ as claimed.
	
	Finally, that the one-homogeneous extension of $\tilde{\varphi}$ is convex follows from Remark 3.6 in \cite{periodic_paper}.  Since \eqref{E: derivation} (or \eqref{E: planar convergence}) does not depend on our choice of $q$, neither does $\tilde{\varphi}$.
\end{proof}  

The only thing left to do is prove Proposition \ref{P: planar_convergence}:

\begin{proof}[Proof of Proposition \ref{P: planar_convergence}]  First, we show
\begin{equation} \label{E: upper bound planar}
\tilde{\varphi}(e) r^{d - 1} \geq \lim_{\delta \to 0^{+}} \limsup_{\epsilon \to 0^{+}} \inf \left\{ \mathcal{F}^{\omega}_{\epsilon}(u; A) \, \mid \, \|u - T_{s_{0} e}\chi_{e}\|_{L^{1}(A)} < \delta \right\}.
\end{equation}  

For each $n \in \mathbb{N}$, define $A_{n} = Q(x_{0},r) \oplus_{e} (s_{0} - n^{-1}t, s_{0} + n^{-1}t)$ and let $(u_{\epsilon}^{(n)})_{\epsilon > 0} \subseteq H^{1}(A_{n}; [-1,1])$ satisfy
\begin{equation*}
\mathcal{F}^{\omega}_{\epsilon}(u_{\epsilon}^{(n)}; A_{n}) = \tilde{\Phi}^{\omega}(e, x_{0} \oplus_{e} s_{0}, A_{n})
\end{equation*}
and $u_{\epsilon}^{(n)} = T_{s_{0}e}q^{\epsilon}_{e}$ on $\partial A_{n}$ for each $\epsilon > 0$.  Extend the functions $(u_{\epsilon}^{(n)})_{\epsilon > 0}$ so that $u_{\epsilon}^{(n)} = T_{s_{0}e} q^{\epsilon}_{e}$ in $\mathbb{R}^{d} \setminus A_{n}$.    

Now observe that
\begin{equation*}
\lim_{\epsilon \to 0^{+}} \|u_{\epsilon}^{(n)} - T_{s_{0}e}\chi_{e}\|_{L^{1}(A)} \leq 2 n^{-1} r^{d - 1}t
\end{equation*}
Moreover, since $\omega \in \tilde{\Omega}$,
\begin{equation*}
\lim_{\epsilon \to 0^{+}} \mathcal{F}^{\omega}_{\epsilon}(u_{\epsilon}^{(n)} ; A) = \lim_{\epsilon \to 0^{+}} \mathcal{F}^{\omega}_{\epsilon}(u_{\epsilon}^{(n)}; A_{n}) + \lim_{\epsilon \to 0^{+}} \mathcal{F}^{\omega}_{\epsilon}(T_{s_{0} e}q^{\epsilon}_{e} ; A \setminus A_{n}) = \tilde{\varphi}(e) r^{d - 1}.
\end{equation*}
Thus, \eqref{E: upper bound planar} follows in the limit $n \to \infty$.    

To prove the lower bound, we use the fundamental estimate.  Fix $(\delta_{j})_{j \in \mathbb{N}} \subseteq (0,\infty)$ such that $\lim_{j \to \infty} \delta_{j} = 0$ and
	\begin{align*}
	\lim_{j \to \infty} \liminf_{\epsilon \to 0^{+}} \inf \left\{ \mathcal{F}^{\omega}_{\epsilon}(u; A) \, \mid \, \|u - T_{s_{0}e}\chi^{e}\|_{L^{1}(A)} < \delta_{j} \right\} &= \text{RHS of} \, \, \eqref{E: planar convergence}
	\end{align*}

For each $j$, pick an $\epsilon_{j} \in (0,\delta_{j}]$ and a function $v_{\epsilon_{j}}$ such that
\begin{align*}
&\mathcal{F}_{\epsilon_{j}}^{\omega}(v_{\epsilon_{j}}; A) \leq \liminf_{\epsilon \to 0^{+}} \inf \left\{ \mathcal{F}_{\epsilon}^{\omega}(u; A) \, \mid \, \|u - T_{s_{0} e}\chi^{e}\|_{L^{1}(A)} < \delta_{j} \right\} + \frac{1}{j} \\
&\|v_{\epsilon_{j}} - T_{s_{0}e}\chi^{e}\|_{L^{1}(A)} < \delta_{j}.
\end{align*}
Note that we make no requirements about boundary conditions.

Fix $\alpha \in (0,1)$ and let $U_{\alpha} = Q(x_{0},\alpha r) \oplus_{e} (s_{0} - \alpha t, s_{0} + \alpha t)$, $U' = A$, and $V_{\alpha} = U' \setminus \overline{U_{\alpha}}$.  By the fundamental estimate, there is a family of cut-off functions $(\psi_{\epsilon_{j}})_{j \in \mathbb{N}} \subseteq C^{\infty}_{c}(U'; [0,1])$ such that $\psi_{\epsilon_{j}} \equiv 1$ on $U_{\alpha}$ and
\begin{equation*}
\mathcal{F}^{\omega}_{\epsilon_{j}}(\psi_{\epsilon_{j}} v_{\epsilon_{j}} + (1 - \psi_{\epsilon_{j}}) T_{s_{0}e}q^{\epsilon_{j}}_{e}; A) \leq \mathcal{F}^{\omega}_{\epsilon_{j}}(v_{\epsilon_{j}}; A) + \mathcal{F}^{\omega}_{\epsilon_{j}}(T_{s_{0}e}q^{\epsilon_{j}}_{e} ; V_{\alpha}) + o(1)
\end{equation*}
as $j \to \infty$.  
Keeping $\alpha$ fixed, letting $j \to \infty$, and appealing to the fact that $\psi_{\epsilon_{j}} v_{\epsilon_{j}} + (1 - \psi_{\epsilon_{j}}) T_{s_{0}e}q^{\epsilon_{j}}_{e}$ equals $T_{s_{0}e}q^{\epsilon_{j}}_{e}$ on $\partial A$, we find
\begin{equation*}
\tilde{\varphi}(e) r^{d - 1} \leq \liminf_{j \to \infty} \mathcal{F}^{\omega}_{\epsilon_{j}}(v_{\epsilon_{j}}; A) + \limsup_{j \to \infty} \mathcal{F}^{\omega}_{\epsilon_{j}}(T_{s_{0}e}q^{\epsilon_{j}}_{e}; V_{\alpha}).
\end{equation*}
Since $\limsup_{\alpha \to 1^{-}} \limsup_{\epsilon \to 0^{+}} \mathcal{F}^{\omega}_{\epsilon}(T_{s_{0}e}q^{\epsilon}_{e}; V_{\alpha}) = 0$, we send $\alpha \to 1^{-}$ to obtain the lower bound in \eqref{E: planar convergence}:
\begin{equation*}
\tilde{\varphi}(e) r^{d - 1} \leq \lim_{\delta \to 0^{+}} \liminf_{\epsilon \to 0^{+}} \inf \left\{ \mathcal{F}^{\omega}_{\epsilon}(u; A) \, \mid \, \|u - T_{s_{0}e}\chi^{e}\|_{L^{1}(A)} < \delta \right\}.
\end{equation*}
Together with \eqref{E: upper bound planar}, this implies \eqref{E: planar convergence} holds.  \end{proof}

\appendix 

\section{The Fundamental Estimate of $\Gamma$-Convergence} \label{A: fundamental_estimate}

We state and prove the version of the fundamental estimate of $\Gamma$-convergence used throughout the paper.  For problems of this type, the fundamental estimate was originally established in \cite{periodic_paper}. 

\begin{theorem}  Assume that the families $(u_{\epsilon})_{\epsilon > 0}, (v_{\epsilon})_{\epsilon > 0} \subseteq H^{1}_{\text{loc}}(\mathbb{R}^{d}; [-1,1])$ and bounded open sets $(U_{\epsilon})_{\epsilon > 0}, (V_{\epsilon})_{\epsilon > 0}, U' \subseteq \mathbb{R}^{d}$ satisfy the following conditions:
\begin{itemize}
\item[(i)] $C_{0} := \sup \left\{ \mathcal{F}^{\omega}_{\epsilon}(u_{\epsilon}; U') + \mathcal{F}^{\omega}_{\epsilon}(v_{\epsilon}; V_{\epsilon}) \, \mid \, \epsilon > 0\right\} < \infty$
\item[(ii)] $U_{\epsilon} \subseteq U$ and $V_{\epsilon} \subseteq V$ for all $\epsilon > 0$, and
\begin{align*}
D &:= \inf \left\{ \text{dist}(U_{\epsilon},\partial U') \, \mid \, \epsilon > 0\right\} > 0 \\
R &:= \sup \left\{\text{diam}(V_{\epsilon}) \, \mid \, \epsilon > 0\right\} < \infty.
\end{align*}
\item[(iii)] There is a $v \in L^{1}_{\text{loc}}(\mathbb{R}^{d}; \{-1,1\})$ such that $\lim_{\epsilon \to 0^{+}} \|v_{\epsilon} - v\|_{L^{1}(V_{\epsilon})} = 0$.
\end{itemize}
If there is a measurable set $E \subseteq \mathbb{R}^{d}$ such that $\mathcal{L}^{d}(E) \leq \zeta$ and $\|u_{\epsilon} - v_{\epsilon}\|_{L^{1}(V_{\epsilon} \setminus E)} \to 0$ as $\epsilon \to 0^{+}$, then there is a constant $C > 0$ depending only on $D$, $\lambda$, $\Lambda$, and $W$ and a family of cut-off functions $(\psi_{\epsilon})_{\epsilon > 0} \subseteq C^{\infty}_{c}(U';[0,1])$ satisfying $\psi_{\epsilon} \equiv 1$ on $U_{\epsilon}$ such that
\begin{equation*}
\mathcal{F}^{\omega}_{\epsilon}(\psi_{\epsilon} u_{\epsilon} + (1 - \psi_{\epsilon}) v_{\epsilon}; U_{\epsilon} \cup V_{\epsilon}) \leq \mathcal{F}^{\omega}_{\epsilon}(u_{\epsilon}; U') + \mathcal{F}^{\omega}_{\epsilon}(v_{\epsilon}; V_{\epsilon}) + C \zeta + o(1).
\end{equation*}  
\end{theorem}  

We remark that we sometimes apply the fundamental estimate with $\epsilon$-independent open sets $U$, $U'$, and $V$.  

\begin{proof}  Fix $\epsilon > 0$.  Let $N_{\epsilon} = \lceil \epsilon^{-1} \rceil$ and pick open sets $U_{1}^{\epsilon},U_{2}^{\epsilon},\dots,U_{N_{\epsilon}}^{\epsilon}$ such that
\begin{itemize}
\item[(1)] $U_{1}^{\epsilon} \Subset U_{2}^{\epsilon} \Subset \dots \Subset U_{N_{\epsilon}}^{\epsilon}$
\item[(2)] $U_{1}^{\epsilon} = U_{\epsilon}$ and $U_{N_{\epsilon}}^{\epsilon} = U'$
\item[(3)] $\text{dist}(U_{i}^{\epsilon}, \partial U_{i + 1}^{\epsilon}) \geq \frac{D}{N_{\epsilon}}$ for each $i \in \{1,2,\dots,N_{\epsilon} - 1\}$
\end{itemize}
For each $i \in \{1,2,\dots,N_{\epsilon} - 1\}$, pick $\psi_{i} \in C^{\infty}_{c}(U_{i + 1}^{\epsilon};[0,1])$ such that $\psi^{\epsilon}_{i} \equiv 1$ in $U_{i}^{\epsilon}$ and
\begin{equation*}
\|D\psi_{i}\|_{L^{\infty}(U_{i + 1}^{\epsilon})} \leq \frac{2 N_{\epsilon}}{D}.
\end{equation*} 
For convenience, write $w^{\epsilon}_{i} = \psi^{\epsilon}_{i} u_{\epsilon} + (1 - \psi^{\epsilon}_{i})v_{\epsilon}$.  

For a fixed $i$, we can write
\begin{align*}
\mathcal{F}^{\omega}_{\epsilon}(w_{i}^{\epsilon}; U_{\epsilon} \cup V_{\epsilon}) &\leq \mathcal{F}^{\omega}_{\epsilon}(u_{\epsilon}; U') + \mathcal{F}^{\omega}_{\epsilon}(v_{\epsilon}; V_{\epsilon}) + \mathcal{F}^{\omega}_{\epsilon}(w^{\epsilon}_{i}; (U_{i + 1}^{\epsilon} \setminus U_{i}^{\epsilon}) \cap V_{\epsilon})
\end{align*}
Appealing to the definitions, we estimate the error term as
\begin{equation*}
\mathcal{F}^{\omega}_{\epsilon}(w^{\epsilon}_{i}; (U_{i + 1}^{\epsilon} \setminus U_{i}^{\epsilon}) \cap V_{\epsilon}) \leq e_{1}(i,\epsilon) + e_{2}(i,\epsilon) + e_{3}(i,\epsilon) 
\end{equation*}
where
\begin{align*}
e_{1}(i,\epsilon) &= \epsilon \Lambda \int_{(U_{i + 1}^{\epsilon} \setminus U_{i}^{\epsilon}) \cap V_{\epsilon}} \left(|Du_{\epsilon}(x)|^{2} + |Dv_{\epsilon}(x)|^{2}\right) \, dx \\
e_{2}(i,\epsilon) &= \epsilon^{-1} \int_{(U_{i + 1}^{\epsilon} \setminus U_{i}^{\epsilon}) \cap V_{\epsilon}} W(w^{\epsilon}_{i}(x)) \, dx \\
e_{3}(i,\epsilon) &= \epsilon \Lambda \left(\frac{2 N_{\epsilon}}{D} \right)^{2} \int_{(U_{i + 1}^{\epsilon} \setminus U_{i}^{\epsilon}) \cap V_{\epsilon}} (u_{\epsilon}(x) - v_{\epsilon}(x))^{2} \, dx
\end{align*}
Summing over $i$, we find 
\begin{align*}
\sum_{i = 1}^{N_{\epsilon} - 1} \mathcal{F}^{\omega}_{\epsilon}(w^{\epsilon}_{i}; (U_{i + 1}^{\epsilon} \setminus U_{i}^{\epsilon}) \cap V_{\epsilon}) &\leq \frac{\Lambda}{\lambda} C_{0} + \epsilon \Lambda  \left(\frac{2 N_{\epsilon}}{D} \right)^{2} \int_{U' \cap V_{\epsilon}} (u_{\epsilon}(x) - v_{\epsilon}(x))^{2} \ dx \\
		&\quad + \epsilon^{-1} \sum_{i = 1}^{N_{\epsilon} - 1} \int_{(U_{i + 1}^{\epsilon} \setminus U_{i}^{\epsilon}) \cap V_{\epsilon}} W(w_{i}^{\epsilon}(x)) \, dx
\end{align*}
Thus, there is a $j_{\epsilon} \in \{1,2,\dots,N_{\epsilon} - 1\}$ such that
\begin{align*}
\mathcal{F}^{\omega}_{\epsilon}(w^{\epsilon}_{j_{\epsilon}}; (U_{j_{\epsilon} + 1}^{\epsilon} \setminus U_{j_{\epsilon}}^{\epsilon}) \cap V_{\epsilon}) &\leq  \epsilon \Lambda (N_{\epsilon} - 1)^{-1} \left(\frac{2 N_{\epsilon}}{D}\right)^{2} \int_{V'} (u_{\epsilon}(x) - v_{\epsilon}(x))^{2} \, dx \\
	&\quad + \epsilon^{-1} (N_{\epsilon} - 1)^{-1} \sum_{i = 1}^{N_{\epsilon} - 1} \int_{(U^{\epsilon}_{i + 1} \setminus U^{\epsilon}_{i}) \cap V_{\epsilon}} W(w_{i}^{\epsilon}(x)) \, dx \\
	&\quad + \frac{\Lambda}{\lambda} C_{0} (N_{\epsilon} - 1)^{-1}
\end{align*}
Observe that there is a constant $C_{1} > 0$ depending only on $D$, $\lambda$, and $\Lambda$ such that
\begin{equation*}
\max\left\{\epsilon \Lambda (N_{\epsilon} - 1)^{-1} \left(\frac{2 N_{\epsilon}}{D}\right)^{2},  \epsilon^{-1} (N_{\epsilon} - 1)^{-1} \right\} \leq C_{1}.
\end{equation*}
Moreover, by assumption,
\begin{equation*}
\limsup_{\epsilon \to 0^{+}} \int_{V_{\epsilon}} \left(u_{\epsilon}(x) - v_{\epsilon}(x)\right)^{2} \, dx \leq 4\zeta.
\end{equation*}
Similarly, since $W$ is continuous on $[-1,1]$ and $\mathcal{L}^{d}(V_{\epsilon}) \leq \omega_{d} R^{d}$ independently of $\epsilon$, we find 
\begin{equation*}
\limsup_{\epsilon \to 0^{+}} \left[ \sum_{i = 1}^{N_{\epsilon} - 1} \int_{(U^{\epsilon}_{i + 1} \setminus U^{\epsilon}_{i}) \cap V_{\epsilon}} W(w_{i}^{\epsilon}(x)) \, dx \right] \leq \|W\|_{L^{\infty}([-1,1])} \zeta.
\end{equation*}
Therefore, as $\epsilon \to 0^{+}$,
\begin{equation*}
\mathcal{F}^{\omega}_{\epsilon}(w_{j_{\epsilon}}^{\epsilon}; U_{\epsilon} \cup V_{\epsilon}) \leq \mathcal{F}^{\omega}_{\epsilon}(u_{\epsilon}; U') + \mathcal{F}^{\omega}_{\epsilon}(v_{\epsilon}; V_{\epsilon}) + C_{1}(4 + \|W\|_{L^{\infty}([-1,1])}) \zeta + o(1).
\end{equation*}
The theorem follows by setting $\psi_{\epsilon} = \psi^{\epsilon}_{j_{\epsilon}}$ and $C = C_{1}(4 + \|W\|_{L^{\infty}([-1,1])})$.  
\end{proof} 

\section{Elements of Ergodic Theory} \label{A: ergodic theory}

\subsection{Conditional Expectation and Translations}  In the proof of the thermodynamic limit, the following lemma was used.  In the sequel, if $X$ is a random variable on $(\Omega,\mathscr{B},\mathbb{P})$ and $\mathcal{G} \subseteq \mathcal{F}$ is a sub-$\sigma$-algebra, we let $\mathbb{E}(X \mid \mathcal{G})$ denote some fixed representative of the conditional expectation of $X$ with respect to $\mathcal{G}$.  Recall that $\mathbb{E}(X \mid \mathcal{G})$ is $\mathcal{G}$-measurable by definition, and it is unique up to almost sure equivalence.  

\begin{lemma} \label{L: shifting and conditioning} If $X$ is a random variable on $(\Omega,\mathscr{B},\mathbb{P})$, $e \in S^{d -1}$, and $x \in \langle e \rangle^{\perp}$, then 
\begin{equation*}
\mathbb{E}(X \circ \tau_{x} \, \mid \, \Sigma_{e}) = \mathbb{E}(X \, \mid \, \Sigma_{e}) \quad \mathbb{P}\text{-almost surely.}
\end{equation*}
\end{lemma}  

\begin{proof}  Recall that $A \in \Sigma_{e}$ if and only if $1_{A} \circ \tau_{x} = 1_{A}$ no matter the choice of $x \in \langle e \rangle^{\perp}$.  Thus, for such an $A$ and $x$, we find
\begin{align*}
\mathbb{E}(X \circ \tau_{x} : A) &= \mathbb{E}((X \circ \tau_{x}) 1_{A}) \\
	&= \mathbb{E}((X \circ \tau_{x})(1_{A} \circ \tau_{x})) \\
	&=\mathbb{E}((X 1_{A})\circ \tau_{x}) \\
	&= \mathbb{E}(X 1_{A}) \\
	&= \mathbb{E}(X : A).
\end{align*}
Since $A$ was arbitrary, uniqueness implies $\mathbb{E}(X \circ \tau_{x} \, \mid \, \Sigma_{e}) = \mathbb{E}(X \, \mid \, \Sigma_{e})$ almost surely.  \end{proof}  

\subsection{Sub-additive ergodic theorem}  Since we will be applying this theorem in a somewhat unconventional ``non-ergodic" form, we will state it carefully and give most of the details of the proof.  

We follow Dal Maso and Modica \cite{nonlinear stochastic homogenization} in the following definition:

\begin{definition} \label{D: sub-additive_process}  Given $e \in S^{d - 1}$, a random function $\psi^{\omega} : \mathcal{U}_{0}^{k} \to \mathbb{R}$ is a sub-additive process over $e$ if it satisfies the following conditions:
\begin{itemize}
\item[(i)] If $A, A_{1},\dots,A_{N} \in \mathcal{U}_{0}^{d - 1}$, $A_{i} \subseteq A$ for each $i$, $A_{i} \cap A_{j}$ for distinct $i,j$, and $\mathcal{L}^{d - 1}(A \setminus (\cup_{j = 1}^{N} A_{j})) = 0$, then
\begin{equation*}
\psi^{\omega} (A) \leq \sum_{i = 1}^{N} \psi^{\omega}(A_{i}).
\end{equation*}
\item[(ii)] If $A \in \mathcal{U}_{0}^{d - 1}$ and $x \in \langle e \rangle^{\perp}$, then 
\begin{equation*}
\psi^{\omega}(A + x) = \psi^{\tau_{x}\omega}(A).
\end{equation*}
\end{itemize}
\end{definition}

Here is the version of the sub-additive ergodic theorem we will use.

\begin{theorem} \label{T: sub-additive_ergodic_theorem}  Suppose $\psi^{\omega}$ is a sub-additive process over $e$ and there is a constant $C_{\psi} > 0$ such that
\begin{equation} \label{E: bounded_process}
0 \leq \psi^{\omega}(A) \leq C_{\psi} \mathcal{L}^{d - 1}(A) \quad \text{if} \, \, A \in \mathcal{U}_{0}^{d - 1}.
\end{equation}
Define $\overline{\psi}^{\omega}$ and $\underline{\psi}^{\omega}$ by
\begin{align*}
\overline{\psi}^{\omega} &= \limsup_{R \to \infty} R^{1 - d} \psi^{\omega}(Q(0,R)) \\
\underline{\psi}^{\omega} &= \liminf_{R \to \infty} R^{1 - d} \psi^{\omega}(Q(0,R)).
\end{align*}
Then $\overline{\psi}^{\omega},\underline{\psi}^{\omega}$ are $\Sigma_{e}$-measurable random variables and $\overline{\psi}^{\omega} = \underline{\psi}^{\omega}$ almost surely.  Moreover, there is an event $\Omega_{\psi} \in \Sigma_{e}$ such that $\mathbb{P}(\Omega_{\psi}) = 1$ and if $\omega \in \Omega_{\psi}$ and $Q$ is a cube in $\mathbb{R}^{d - 1}$, then
\begin{equation}
\underline{\psi}^{\omega} \mathcal{L}^{d - 1}(Q) = \lim_{R \to \infty} R^{1- d} \psi^{\omega}(RQ).
\end{equation}
\end{theorem}  

The fact that $\overline{\psi}$, $\underline{\psi}$, and $\Omega_{\psi}$ are $\Sigma_{e}$-measurable is used in Lemma \ref{L: helpful_ergodic} and Theorem \ref{T: thermodynamic_limit}.  

\begin{proof}  First, we claim that 
\begin{align*}
\overline{\psi}^{\omega} &= \limsup_{\mathbb{N} \ni N \to \infty} N^{1 - d} \psi^{\omega}(Q(0,N)) \\
\underline{\psi}^{\omega} &= \liminf_{\mathbb{N} \ni N \to \infty} N^{1 - d} \psi^{\omega}(Q(0,N)).
\end{align*}
Indeed, this follows immediately from \eqref{E: bounded_process} and condition (i) in the definition of sub-additive process.  If $N = \lfloor R \rfloor$, then
\begin{align*}
(N + 1)^{1- d} \psi^{\omega}(Q(0,N + 1)) - C \left\{1 - \left(\frac{R}{N + 1}\right)^{d - 1}\right\} \leq R^{1 - d} \psi^{\omega}(Q(0,R)) \\
R^{1 - d} \psi^{\omega}(Q(0,R)) \leq N^{1 - d} \psi^{\omega}(Q(0,N)) +  C \left\{1 - \left(\frac{N}{R}\right)^{d - 1}\right\}.
\end{align*}
Thus, the limit inferior and limit superior do not change if we restrict $R$ to $\mathbb{N}$.  

Now we claim that $\overline{\psi}^{\omega}$ and $\underline{\psi}^{\omega}$ are $\Sigma_{e}$-measurable.  This follows from condition (ii) in the definition of sub-additive process.  Suppose $x \in \langle e \rangle^{\perp}$.  Notice that $Q(0,N) \subseteq Q(x,N + |x|_{\infty})$.  Moreover, $\frac{N + |x|_{\infty}}{N} \to 1$ as $N \to \infty$.  Using \eqref{E: bounded_process}, we find \begin{equation*}
\psi^{\omega}(Q(x,N + |x|_{\infty})) \leq \psi^{\omega}(Q(0,N)) + C_{\psi} ((N + |x|_{\infty})^{d - 1} - N^{d - 1}).
 \end{equation*}
 Therefore, from the equality $\psi^{\tau_{x}\omega}(Q(0,N + |x|_{\infty})) = \psi^{\omega}(Q(x,N + |x|_{\infty}))$, we deduce that $\overline{\psi}^{\tau_{x} \omega} \leq \overline{\psi}^{\omega}$ and $\underline{\psi}^{\tau_{x}\omega} \leq \underline{\psi}^{\omega}$.   Replacing $x$ with $-x$ yields the complementary inequality.  Thus, $\overline{\psi}^{\omega}$ and $\underline{\psi}^{\omega}$ are $\Sigma_{e}$-measurable. 

Let $\Omega^{(1)} = \{\omega \in \Omega \, \mid \, \overline{\psi}^{\omega} = \underline{\psi}^{\omega}\}$.  
 By the (multi-parameter) sub-additive ergodic theorem (cf.\ \cite{krengel}), $\mathbb{P}(\Omega^{(1)}) = 1$.  In other words, $\overline{\psi}^{\omega} = \underline{\psi}^{\omega}$ almost surely.  
 
 Next, let $\mathcal{D}_{\mathbb{Q}}$ be the family of all open cubes in $\mathbb{R}^{d -1}$ with rational endpoints, that is, $\mathcal{D}_{\mathbb{Q}} = \{Q(x,\rho) \, \mid \, x \in \mathbb{Q}^{d - 1}, \, \, \rho \in \mathbb{Q} \cap (0,\infty)\}$.  Define the event $\Omega^{(2)}$ by
 \begin{equation*}
 \Omega^{(2)} = \bigcap_{Q \in \mathcal{D}_{\mathbb{Q}}} \left\{ \omega \in \Omega \, \mid \, \underline{\psi}^{\omega}\mathcal{L}^{d -1}(Q) = \lim_{\mathbb{N} \ni N \to \infty} N^{1- d} \psi^{\omega}(NQ) \right\}
 \end{equation*}
 The sub-additive ergodic theorem implies $\Omega_{h}^{(2)}$ is a countable intersection of probability one events.  Therefore, $\mathbb{P}(\Omega^{(2)}) = 1$.  Arguing as before, we see that $\Omega^{(2)}$ can alternatively be characterized as follows:
 \begin{equation*}
  \Omega^{(2)} = \bigcap_{Q \in \mathcal{D}_{\mathbb{Q}}} \left\{ \omega \in \Omega \, \mid \, \underline{\psi}^{\omega}\mathcal{L}^{d -1}(Q) = \lim_{R \to \infty} R^{1- d} \psi^{\omega}(RQ) \right\}
   \end{equation*}  Similarly, the arguments used to show $\underline{\psi}^{\omega}$ and $\overline{\psi}^{\omega}$ are $\Sigma_{e}$-measurable can be adapted to prove $\Omega^{(2)} \in \Sigma_{e}$.  
   
 Finally, let $\mathcal{D} = \{Q(x,\rho) \, \mid \, x \in \mathbb{R}^{d - 1}, \, \rho > 0\}$.  We claim that 
 \begin{equation*}
 \Omega^{(2)} = \bigcap_{Q \in \mathcal{D}} \left\{ \omega \in \Omega \, \mid \, \underline{\psi}^{\omega}\mathcal{L}^{d -1}(Q) = \lim_{R \to \infty} R^{1- d} \psi^{\omega}(RQ) \right\}.
 \end{equation*}
 This follows since we can approximate any cube in $\mathcal{D}$ by a cube in $\mathcal{D}_{\mathbb{Q}}$.  As the arguments are similar to the ones already exposed, we omit the details.
 
 The proposition follows with $\Omega_{\psi} = \Omega^{(1)} \cap \Omega^{(2)}$.  
 \end{proof} 
 
 In Section \ref{S: other_cubes}, we will use the following version of Birkhoff's ergodic theorem.  Since we use the fact that the event on which it holds is in $\Sigma$, we provide a complete statement and sketch the proof:
 
\begin{theorem} \label{T: additive_ergodic_theorem}  If $E \in \mathscr{B}$, then there is an $\Omega_{E} \in \Sigma$ satisfying $\mathbb{P}(\Omega_{E}) = 1$ such that if $Q$ is any open cube in $\mathbb{R}^{d}$, then
\begin{equation*}
\mathbb{P}(E) \mathcal{L}^{d}(Q) = \lim_{R \to \infty} R^{-d} \mathcal{L}^{d}(\{y \in RQ \, \mid \, \tau_{y} \omega \in E\}) \quad \text{if} \, \, \omega \in \Omega_{E}.
\end{equation*}  
\end{theorem}

The $\Sigma$-measurability of $\Omega_{E}$ is used in the proof of Proposition \ref{P: other_cubes 1} to show that the event $\hat{\Omega}$ in Theorem \ref{T: main result} is itself $\Sigma$-measurable.  

\begin{proof}  Define a sub-additive process $\psi^{\omega} : \mathcal{U}_{0}^{d} \to \mathbb{R}$ by
\begin{equation*}
\psi^{\omega}(A) = \mathcal{L}^{d}(\{y \in A \, \mid \, \tau_{y} \omega \in E\}).
\end{equation*}
It is straightforward to verify that $\psi^{\omega}$ satisfies Definition \ref{D: sub-additive_process} with $d - 1$ replaced by $d$ and $\langle e \rangle^{\perp}$ replaced by $\mathbb{R}^{d}$.  These changes have no bearing on the proof of Theorem \ref{T: sub-additive_ergodic_theorem} other than switching $\Sigma$ for $\Sigma_{e}$.  

As for the value of the limit, we can see that it should be $\mathbb{P}(E)$ by appealing to the dominated convergence theorem and ergodicity.  \end{proof}

\section*{Acknowledgements}  

The author is grateful to P.E.\ Souganidis for introducing him to this subject and suggesting he revisit some related open problems.  Additionally, he thanks Y.\ Bakhtin and N.\ Dirr for helpful suggestions and encouragement.

The author was partially supported by the National Science Foundation Research Training Group grant DMS-1246999.

 \end{document}